\documentclass{amsart}
\usepackage{amsmath}
\usepackage{amssymb}
\usepackage{color}
\numberwithin{equation}{section}

\newcommand{\CAT}{\operatorname{CAT}}

\newcommand{\DC}{\operatorname{DC}}

\newcommand{\R}{\operatorname{\mathbb{R}}}
\newcommand{\N}{\operatorname{\mathbb{N}}}

 \newcommand{\Sph}{\operatorname{\mathbb{S}}}

\theoremstyle{plain}
\newtheorem{thm}{Theorem}[section]
\newtheorem{lem}[thm]{Lemma}
\newtheorem{prop}[thm]{Proposition}
\newtheorem{cor}[thm]{Corollary}
\newtheorem{defn}[thm]{Definition}
\newtheorem{exmp}[thm]{Example}
\newtheorem{rem}[thm]{Remark}

\begin{document}

\title
[Geodesically complete spaces with an upper curvature bound]
{Geodesically complete spaces \\
with an  upper curvature bound}

\author
[A. Lytchak]{Alexander Lytchak}
\author
[K. Nagano]{Koichi Nagano}

\thanks{The first author was partially supported by the DFG grants   SFB TRR 191 and SPP 2026.
The second author was partially supported
by JSPS KAKENHI Grant Numbers 26610012, 21740036, 18740023,
and by the 2004 JSPS Postdoctoral Fellowships for Research Abroad.}

\email
[Alexander Lytchak]
{alytchak@math.uni-koeln.de}
\email
[Koichi Nagano]
{nagano@math.tsukuba.ac.jp}

\address
[Alexander Lytchak]
{\endgraf
Mathematisches Institut, Universit\"{a}t K\"oln
\endgraf
Weyertal 86-90, D-50937 K\"oln, Germany}
\address
[Koichi Nagano]
{\endgraf
Institute of Mathematics, University of Tsukuba
\endgraf
Tennodai 1-1-1, Tsukuba, Ibaraki, 305-8571, Japan}


\keywords
{Curvature bounds, rectifiability, metric measure spaces, stratification, DC-functions}
\subjclass
[2010]{53C20, 53C21, 53C23}

\begin{abstract}
We study  geometric and topological properties
of  locally compact, geodesically complete spaces with an upper curvature bound.
We control the size of  singular subsets, discuss
homotopical and measure-theoretic stratifications  and  regularity of the
metric structure on a large part.
\end{abstract}

\maketitle


\section{Introduction}
\subsection{Object of investigations}
Metric spaces with  one-sided curvature bounds were introduced  by A.D. Alexandrov in \cite{Aleksandrov}.
 After the revival of
metric geometry in the eighties, properties and applications of such spaces
have been investigated from various points of view, we refer to \cite{ballmannbook},
\cite{BH}, \cite{BBI01},  \cite{BuyaloSchr}, \cite{AKP} and the bibliography therein.
Starting with  \cite{BGP} a structure theory
of locally compact spaces with a lower curvature bound and
finite dimension,  so-called Alexandrov spaces,
was developed, see
\cite{AKP} for the huge bibliography.

 Due to the local uniqueness of geodesics in spaces with upper curvature bounds,
the derivation of basic topological and geometric properties is simpler  than in the case  of Alexandrov spaces. However,
  finer structural features  can be  much more intricate. Even a compact tree, hence a topologically $1$-dimensional
 space of non-positive curvature, can have infinite Hausdorff dimension and may not contain any kind of manifold charts, \cite{Kleiner}, \cite{Berest-path}.
 Also the global topological structure of spaces with upper curvature bounds can be much more complicated than in the case of lower curvature bounds:  for instance,  any finite-dimensional simplicial complex carries a
 metric with an upper  curvature bound    \cite{Berest}.

 Without additional assumptions it seems  impossible to detect
some general regular structures beyond  a theorem of B. Kleiner, \cite{Kleiner},
claiming that the topological dimension coincides with the maximal dimension of a   topological manifold embedded into the space.
  In order to obtain some control,  one needs an assumption,  which  provides a close relation between the local
geometry near a point and  the   geometry of its tangent cone.
Such a natural assumption is  \emph{(local) geodesic completeness}, also known as  \emph{geodesic extension property}.
  It says
that any compact  geodesic  can be extended as a local geodesic beyond its endpoints.
This condition is stable
under natural metric operations and
can often be  detected  topologically. For example, it holds
if  small
punctured neighborhoods of all points are non-contractible.
Finally, geodesic completeness plays an important  role  in geometric group
theory, see, for instance,   \cite{CM}, \cite{Genevois}.

The present paper is devoted to the description of basic measure-theoretic, homotopic and analytic properties
of such spaces, recovering analogues of most results of  \cite{BGP}, \cite{P3} and \cite{OS}. Applications to topological questions,
 geometric group theory and sphere theorems will be discussed in forthcoming papers. Results and ideas of preliminary versions of this  work which we have circulated in the last 10 years have already been used, for instance in \cite{Kap}, \cite{Linus}, \cite{Bert}, \cite{KK}.

There has  been one systematic investigation of the theory of geodesically complete spaces with upper curvature bounds by Otsu and Tanoue,  \cite{Otsu},  announced in \cite{Ots}
(and continued in \cite{N2}).
Since \cite{Otsu} has never been published and is rather difficult to read, we do not use  it. In fact, some of our theorems
provide improvements and simplifications of the main
results from \cite{Otsu}.

 The special case of  two-dimensional   topological surfaces has been intensively studied,
  see   \cite{Reshetnyak-GeomIV}. Some results from \cite{Reshetnyak-GeomIV}, definitely out  of reach in  higher dimensions, have been generalized to  two-dimensional polyhedra in \cite{BB}.

\subsection{Main results}
From now on,  we   say that  $X$ is  \emph{GCBA},
if $X$  is  a separable, locally compact,   locally geodesically complete space  with  curvature bound above.

GCBA spaces  have indeed many structural
similarities with   Alexandrov spaces, see Section \ref{sec:setting}.
Any GCBA space $X$ is locally doubling. 
For any $x$ in $X$, the tangent space  $T_xX$ and the space of directions $\Sigma _x X$ are again  GCBA.  Any compact part
of any GCBA admits a biLipschitz embedding into a Euclidean space.

The following theorem  is implicitly contained  in \cite{Otsu},  compare also \cite{Kleiner}.

\begin{thm} \label{thm0}
 Let $X$ be GCBA.  The topological dimension $\dim (X)$ of    $X$  coincides with the Hausdorff dimension.
It equals  the  supremum of   dimensions of  open subsets of $X$ homeomorphic to  Euclidean balls.
\end{thm}

The local  dimension might  be non-constant on $X$, as  in the case of simplicial complexes. But the local dimension
  can be understood by looking at  tangent spaces.
For $k=0,1,2,...$, we call \emph{the $k$-dimensional part of $X$}, denoted by $X^k$,  the set of all points $x\in X$ with $\dim (T_xX)= k$.
In general,  $X^k$ is neither open nor closed in $X$.  However, $X^k$ contains large  "regular subsets",  open in $X$, as shown by the next result.

\begin{thm} \label{thm:locdim}
Let $X$ be  GCBA. A point $x$ is contained in the $k$-dimensional part $X^k$ if and only if all sufficiently small balls around $x$ have dimension $k$.
The Hausdorff measure $\mathcal H^k$ is locally finite and locally positive on $X^k$.  There is a subset  $M^k$ of $X^k$, which is open in $X$, dense in $X^k$ and
locally biLipschitz equivalent to $\R^k$.   Moreover,  the complement
 $\bar X^k\setminus M^k$ of $M^k$ in the closure  $\bar X^k$ of $X^k$  has Hausdorff dimension at most $k-1$.
\end{thm}

We refer to Section \ref{sec:dim} for a  stronger  statement.  The open manifold $M^k$ should be thought of as the regular $k$-dimensional part of $X$. Its finer geometry is described by the following theorem. We refer to \cite{P3}, \cite{Shioya}, \cite{Ambr} and Section \ref{sec:DC} below for a discussion of  the notions of $\DC$-functions, $\DC$-manifolds and  functions of bounded variation  used in the next theorem.

\begin{thm} \label{thm:regular}
For $k=0,1,...$,  the manifold $M^k\subset X^k$ in Theorem \ref{thm:locdim} can be chosen to satisfy the following properties.
  The set $M^k$   contains the set $\mathcal R_k$ of all points  $x\in X$ with tangent
  space $T_xX$ isometric to $\R^k$.  The  Hausdorff dimension of
  $M^k\setminus \mathcal R_k$ is at most $ k-2$.
The manifold $M^k$ has a
unique $\DC$-atlas such that all convex
functions on $M^k$ are $\DC$-functions with respect to this atlas.
 The distance
in $M^k$ can locally  be  obtained from  a Riemannian metric tensor  $g$,  well-defined  and continuous on $\mathcal R_k$.
The  tensor $g$  locally is of  bounded variation  on $M^k$.
\end{thm}

The $k$-dimensional Hausdorff measure is  the
 natural measure on the $k$-dimensional part  $X^k$ of $X$. We put these measures together and define the \emph{canonical measure
 $\mu _X$ }  of the space $X$ to be the sum of the restrictions of $\mathcal H^k $ to $X^k$,  thus
$$\mu _X :=\sum _{k=0} ^{\infty} \mathcal H^k \llcorner X^k \, .$$
 By Theorem \ref{thm:locdim} and Theorem \ref{thm:regular}, the restriction of $\mu _X$  to $M^k$ is (the Riemannian measure)   $\mathcal H ^k$, and $\mu _X$
  vanishes on the complement of the open submanifold $\bigcup _{k\geq 0} M^k$.
 The "canonicity" of $\mu _X$  is confirmed by the following two theorems.

\begin{thm} \label{thm:finitecan}
 Let $X$ be  GCBA. The canonical measure is positive and finite on any open relatively compact subset of $X$.
 \end{thm}

The second theorem tells us that the canonical measure is continuous  with respect to the Gromov--Hausdorff topology. We formulate it here  for compact spaces and refer
to Section \ref{sec:stablecan} for the general  local statement.

\begin{thm}  \label{thm:stablecan}
Let $X_l$ be a sequence of compact GCBA spaces of dimension, curvature and diameter bounded  from above and  injectivity radius bounded from below by some constants. The total measures $\mu _{X_l} (X_l)$ are bounded from above by a
 constant if and only if,  upon choosing a subsequence, $X_l$ converge to a compact GCBA  space $X$ in the Gromov--Hausdorff topology.
 In this case,    $\mu _{X_l} (X_l)$ converge to   $\mu _{X} (X)$.
\end{thm}

Having described the regular parts of $X$ we turn to a stratification of the singular parts $X^k\setminus M^k$ neglected by the canonical measure.
 The following
stratification of $X$, a weak surrogate of the topological
stratification, is motivated   by the example of skeletons   of a simplicial complex.

 For a natural number $k$, we say that a point $x\in X$ is  \emph{$(k,0)$-strained} if its tangent space
$T_xX$ admits  the Euclidean space $\R ^k$  as a direct  factor. We denote by $X_{k,0}$
the set of all $(k,0)$-strained points in $X$.

\begin{thm} \label{thmfirst}
Let $X$ be GCBA and $k\in \N$.   Then $X \setminus X_{k,0}$
is a countable union of  subsets, which are  biLipschitz equivalent to some compact subsets of  $\R^{k-1}$.
\end{thm}

In particular,  the set of not $(k,0)$-strained points is \emph{countably $(k-1)$-rectifiable}.  For similar rectifiable stratifications on different classes of metric spaces we refer, for instance, to \cite{Mondino} and the literature therein.

 \begin{rem}
Parts of  Theorems \ref{thm:locdim},  \ref{thm:regular}, \ref{thm:stablecan}, \ref{thmfirst} have   analogues   in  \cite{Otsu} and   \cite{N2}. We refer to Remarks \ref{rem:remthmfirst},   \ref{rem:remlocdim},  \ref{rem:remstablecan},
and \ref{rem:hessian} for a   comparison.
\end{rem}

\begin{rem}
 Theorem  \ref{thm:finitecan} provides an answer to a question from  \cite{Lafont}. Moreover, Theorem  \ref{thm:finitecan} implies the validity of the property  (U) from \cite{Leuzinger} on any co-compact GCBA space. Hence, it shows the validity of  the main theorem in \cite{Leuzinger}  for all such spaces, see the discussion in \cite{Lafont}.
\end{rem}

\subsection{Main tool and further results}
We are going to introduce the main tool of the paper and a more  informal description of further central results.
 The set   $X_{k,0}$  is usually not open in $X$. As in the theory of Alexandrov spaces developed in  \cite{BGP},
there is a natural way to open up the condition of being  $(k,0)$-strained.

 For any  $\delta>0$, we define
an open subset $X_{k,\delta}$ of a GCBA space $X$ which consists of \emph{$(k,\delta)$-strained} points. While the definition of being $(k,\delta)$-strained is slightly technical,
see Sections \ref{sec:spherical} and \ref{sec:strainer}, the meaning is very simple.

 A point $x\in X$ is $(k,\delta)$-strained, for a small $\delta$, if its tangent space $T_xX$ is sufficiently close to a  space which splits off a direct $\R^k$-factor, see  Proposition  \ref{lem:equiv}.
In other words, a point $x\in X$ is $(k,\delta )$-strained if and only if there  exist  $k$ points $p_1,...,p_k \in X\setminus \{ x\}$, close to $x$,
such that the following holds true. The geodesics $p_ix$ meet in $x$ pairwise at an angle close to $\pi /2$ and the possible
branching angles of the geodesics $p_ix$ at $x$ are small ("small" and "close" is expressed in terms of $\delta$).

The subsets $X_{k,\delta}$ are open in $X$ and  decrease for fixed $k$ and decreasing $\delta$. The   set $X_{k,0}$ of $(k,0)$-strained points  is the countable intersection  $X_{k,0}=\bigcap _{j=1} ^{\infty}  X_{k,\frac 1 j}$.

Each point $x\in X_{k,\delta}$ comes along with  natural maps,  so-called $(k,\delta)$-\emph{strainer maps} $F:V\to \R^k$, defined on a neighborhood $V$ of $x$.  Strainer maps are   analogues  of the  orthogonal  projection onto a face, defined in a neighborhood of
that face in a simplicial complex.
The coordinates of $F$ are distance functions to points  $p_i$ in $X \setminus \{ x\}$, for a $k$-tuple $(p_i)$ as in the above definition of $(k,\delta)$-strained points. In other words, a point $x\in X$ is $(k,\delta)$-strained if and only if there exists a $(k,\delta)$-strainer map
$F$ on a neighborhood $V$ of $x$.

The basic example of a strainer map, responsible for their abundance, is given by the following observation. For any point $p$ in a GCBA space $X$, and any $\delta >0$,  the distance function to $p$ is   a  $(1,\delta)$-strainer map on a small punctured  ball around  $p$, Proposition \ref{lem: first}.

 For  $\delta$ small enough,
any $(k,\delta)$-strainer map $F$
 is similar to a  Riemannian submersion, see  Sections  \ref{sec:strmaps}, \ref{sec:fibers}.
  In particular, $k$ is not larger than  $\dim (X)$.
The following technical result  is the base for all  further investigations on singular sets:

\begin{thm} \label{thm:technical}
Let  $F:V\to \R^k$ be a $(k,\delta)$-strainer map on a sufficiently small open subset $V$ of
 a GCBA space $X$.   Then
the set $V\setminus X_{k+1, 12 \cdot \delta}$ is  a union of a countable family of compact  subsets $K_i$
 such that
 $F:K_i\to F(K_i)$ is biLipschitz.
\end{thm}

The biLipschitz constant of the restrictions $F_i:K_i\to F(K_i)$ and the total measure $\mathcal H^{k} (V\setminus X _{k+1, 12 \cdot \delta})$ in Theorem \ref{thm:technical} are bounded in terms of
$\delta$ and $V$, see Theorem \ref{lem: +1strainer} below.
   The theorem  allows, by a reverse induction on $k$, a  good control of the measures of  singular sets. We refer to Section \ref{sec:strata} for quantitative
   versions of the volume estimates, leading to proofs (and more precise versions) of   Theorems \ref{thmfirst}, \ref{thm:locdim}, \ref{thm:finitecan}.

 The strainer construction   is stable
under Gromov--Hausdorff limits, Section \ref{sec:strainer}.  This  provides us the basic tool for the proof of Theorems \ref{thm:finitecan} and  \ref{thm:stablecan}.

The relation to Theorem \ref{thm:locdim} is established by defining $M^k$ to be the intersection of the $k$-dimensional part $X^k$ of $X$ with $X_{k,\delta}$ for sufficiently small $\delta$.  The DC-atlas on $M^k$ in Theorem \ref{thm:regular} is provided by the $(k,\delta)$-strainer maps.
\begin{rem}
If $k =\dim (X)$ then $X_{k,\delta}$ is closely related to  sets of \emph{not $\delta'$-branch points} used in \cite{Otsu} to analyze the regular part of a GCBA space.
\end{rem}

 From the  homotopy theoretical point of view,
strainer maps have  simple local structure.
 We refer to Sections \ref{sec:strmaps}, \ref{sec:fibers}  for a more   precise  version of the following
\begin{thm} \label{thm:contractible}
Let $F:V\to \R^k$ be  a  $(k,\delta)$-strainer map with $\delta \leq \frac 1 {20 \cdot k}$.
Then,  the map $F$ is open and  for any compact subset $V'$ of $V$, there is some $\epsilon >0$ with the following property. For any $x\in V'$  and any $0<r<\epsilon$,
the open ball of radius $r$   around $x$ in the fiber  $  F^{-1} (F(x)) $
is contractible.
\end{thm}

In the continuation \cite{LNtop} of the present paper,
this result will be used to prove that strainer maps are local
\emph{Hurewicz  fibrations}.
Here, we just apply the homotopic stability results of \cite{Petersen} and deduce, that
if a fiber $F^{-1} (\mathfrak t)$ is compact in $V$ then all nearby fibers are homotopy equivalent to it.
  Moreover,  this homotopical stability of fibers is preserved under Gromov--Hausdorff convergence, as will follow from   \cite{Petersen}.  We refer to Section \ref{sec:stable} for exact results and state here the following  special case (originating from the convergence of the rescaling of the given space to the tangent cone at a point).

\begin{thm} \label{thmsph}
Let $X$ be  GCBA.
For each point $x\in X$ there is some $r_x>0$ such that for all $r< r_x$ the
metric sphere $\partial B_r (x)$ of radius $r$ around $x$ is homotopy equivalent to the space of directions
$\Sigma _x X$.
\end{thm}

If the metric sphere in Theorem \ref{thmsph} is replaced by a punctured ball, the result
is  simpler and the extendibility of  geodesics does not need to be assumed. This has been observed by Kleiner (unpublished)
and  appeared in \cite{Linus}.

The term homotopy equivalent in Theorem \ref{thmsph}
cannot be replaced by "homeomorphic" (Example \ref{ex:surface} below), as it were the case for Alexandrov spaces,
\cite{P2}, \cite{Kap}. This example  shows that,    for general GCBA spaces,
there is no hope of obtaining
a local conicality theorem or topological stability  as in \cite{P2}.
 Moreover, while strainer maps are local fibrations they do not need to be local fiber bundles.

 In the continuation \cite{LNtop} of this paper, we  prove,  that  in some interesting cases,
 the measure-theoretical and homotopy-theoretical stratifications described above
can be upgraded, to provide control on the topological type of the spaces
 in question.
   For instance, GCBA spaces which arise as  limits of Riemannian manifolds can be very well understood, similarly to    \cite{Kap-reg}.

 \subsection{Structure of the paper} We  are going to describe   the contents and main results  of the sections of the paper.

 In the auxiliary Sections \ref{sec:prel}, \ref{sec:prelcurv}, \ref{sec:extend}  we collect preliminaries about general metric spaces,
  spaces with upper curvature bounds and  the geodesic extension property.

  In Section \ref{sec:gcba} we begin the  investigation
  of the central object of this paper and discuss all  properties of  GCBA spaces not based on the notion of strainers.  We localize all discussions by introducing the notion of a \emph{tiny ball of a GCBA space}, a relatively compact ball of a radius very small in comparison to the  curvature bound.  All later results are proven first inside tiny balls and then by  covering the whole space by tiny balls.  It is shown that tiny balls are doubling and that the bound on the doubling constant  (\emph{the capacity of the tiny ball}) is essentially equivalent to the precompactness
in the Gromov--Hausdorff topology, Propositions \ref{prop:doubl}, \ref{prop:precomp}.  We show that tiny balls
  admit almost isometric embeddings into  finite dimensional normed spaces, Proposition \ref{lem:embed}.
  We show that tangent spaces of GCBA spaces are GCBA,  Corollary \ref{cor:converge},  and describe a natural semicontinuity  of tangent spaces under convergence, Lemma \ref{lem: semiproj}.

 In Sections \ref{sec:spherical}, \ref{sec:strainer}, \ref{sec:strmaps}  we define $(k,\delta)$-strained points, $(k,\delta)$-strainer
 and $(k,\delta )$-strainer maps, Definitions \ref{def:deltasph}, \ref{def:deltaort}, \ref{def:deltastr}, \ref{defn: strainer}.
  We confirm that $(k,\delta)$-strained points are exactly the points whose  tangent space is "close" to a space with  an $\R^ k$-factor, Proposition \ref{lem:equiv}.
  We discuss the abundance  of strainers,
  Proposition \ref{lem: first}, and prove that the  notions are stable under small perturbations.

  In Section \ref{sec:strmaps} we show that $(k,\delta)$-strainer maps with small $\delta$ are almost  submersions,
  Lemma \ref{lem:openstr},  Corollary \ref{cor:openstr}.

  In Section \ref{sec:fibers}  we prove  Theorem \ref{thm:contractible} and discuss an application.

  In Section \ref{sec:strata} we show that no  tiny ball  contains arbitrary large subsets such that
  no point of this subset is a strainer of some  other point of this subset, Proposition \ref{prop: bad}. Based on this result
  we prove generalized versions of Theorems \ref{thmfirst}, \ref{thm:technical}.

   In Section \ref{sec:dim} we prove generalized versions of Theorems \ref{thm0}, \ref{thm:locdim} and \ref{thm:finitecan}.

 Section \ref{sec:stablecan} is devoted to the proof of a generalized version of  Theorem \ref{thm:stablecan}.

  In Section \ref{sec:stable} we prove a generalized version of Theorem \ref{thmsph}.

  In Section \ref{sec:DC} we follow \cite{P3} to prove that  a $(k,\delta)$-strainer map
  on a subset of $X^k$ is a local $\DC$-isomorphism.  From this we deduce Theorem \ref{thm:regular}
  and show that any $\DC$-function is twice differentiable almost everywhere, Proposition \ref{prop:hessian}.
  Of special interest,  in particular, for volume rigidity, \cite{Li},  might be  general  stability of length under convergence of
  $\DC$-curves, Proposition \ref{prop:contin}.

\subsection*{Acknowledgments}  Most results  of this paper  have been obtained and presented in talks more than 10 years ago. The authors would like to express their gratitude to people, who have shown interest  in our results and whose interest was  responsible for the finalization of the paper. In particular, we would like to thank Werner Ballmann,
J\'{e}r\^{o}me Bertrand, Pierre-Emmanuel Caprace, Karsten Grove, Vitali Kapovitch, Beno\^it Kloeckner, Urs Lang,   Takao Yamaguchi.   We are  grateful to Anton Petrunin and Stephan Stadler for very helpful comments.


\section{Preliminaries} \label{sec:prel}

\subsection{Spaces and maps}
\cite{BH}, \cite{BBI01} and \cite{ballmann} are
general references for this section.
By $d$ we  denote distances in metric spaces.
For a subset $A$ of a metric space $X$ and $r>0$, we denote by $B_r (A)$  the \emph{open  neighborhood of radius $r$ around $A$},
hence the set of all points with distance less than $r$ from $A$.
By $ r \cdot  X$ we
denote the set $X$ with the metric rescaled by $r$.
A space is \emph{proper} if its closed bounded subsets are compact.

A subset of a metric space is called \emph{$r$-separated}
if its  elements have pairwise distances at least  $r$.
A  metric space $X$  is \emph{doubling}
(more precisely, $L$-doubling)
if no   ball of radius $r$ in $X$ has an $(r/2)$-separated
subset with  more than $L$ elements.
Equivalently, any $r$-ball is covered by a uniform number of balls of radius $r/2$.

The \emph{length of a curve} $\gamma $ in a metric space is denoted by
$\ell (\gamma )$.
A \emph{geodesic}
is an isometric embedding of an interval.
A \emph{triangle} is a union of three geodesics connecting three points.
A \emph{local geodesic} is a curve $\gamma :I\to X$ in a metric space $X$ defined on an  interval
$I$, such that the restriction of $\gamma$ to a small neighborhood of any  $t\in I$ is a geodesic.
$X$  is \emph{a  geodesic metric space} if
any pair of   points of $X$
is connected by a geodesic.

A map $F:X\to Y$ between metric spaces is called \emph{$L$-Lipschitz} if
$d(F(x),F(\bar x))\leq L \cdot d(x,\bar x)$,  for all $x,\bar x \in X$.
 A map $F:X\to Y$ is called an \emph{ $L$-biLipschitz embedding} if, for all
$x,\bar x \in X$, one has
$\frac 1 L \cdot  d(x,\bar x) \leq d(F(x),F(\bar x))\leq L \cdot d(x,\bar x)$.

Let $Z$ be a   metric space and $C>0$. A continuous  map $F:Z\to Y$ is called
\emph{$C$-open} if the following condition holds.
For any   $z\in Z$  and any $r>0$ such  that the closed ball  $\bar B_{Cr} (z)$ is complete,
 we have the inclusion
$B_{r} (F(z)) \subset F(B_{Cr} (z))$.

A function $f:X\to \R$ on a metric space $X$ is \emph{convex} if its restriction $f\circ \gamma$
to any geodesic $\gamma :I\to X$ is a convex function on the interval $I$.

\subsection{Convergence} \label{subsec:conv}
On the set of isometry classes of compact metric spaces we will use the Gromov--Hausdorff distance.
By an abuse of definition we will identify spaces and their isometry classes.
Whenever spaces $X,Y$ at  Gromov--Hausdorff distance smaller $\delta$ appear, we will  implicitly assume that isometric embeddings  $f:X\to Z$ and $g:Y\to Z$ into some
  metric space $Z$ are fixed such that the Hausdorff distance between $f(X)$ and $g(Y)$ is smaller than $2\delta$.

On the set of isometry classes of pointed proper metric spaces
we will consider
the  pointed Gromov--Hausdorff topology
(abbreviated as GH-topology),
and denote by $(X_l,x_l) \to (X,x)$
a convergent sequence.
Each sequence of doubling spaces with a uniform doubling constant
has a subsequence converging in
the GH-topology.
The limit space is doubling with the same doubling
constant.

It is often simpler to work with \emph{ultralimits} instead of GH-limits. There are several advantages: the ultralimits are always defined,
the limit object is a space and not just an "isometry class" and there is no need to consider subsequences, see
  \cite{AKP} for details.
We fix a  non-principal  ultrafilter  $\omega$  on $\N$  and  denote by $\lim_{\omega} (X_i,x_i)$ the $\omega$-ultralimit of pointed metric spaces $(X_i,x_i)$.
For $C$-Lipschitz maps $f_l:(X_l,x_l)\to (Y_l,y_l)$ we will denote by  $\lim _{\omega} f_l$ the  ultralimit  map
$f:(X,x)\to (Y,y)$.

Whenever proper spaces $(X_l,x_l)$ converge in the GH-topology to $(X,x)$, the ultralimit $\lim _{\omega} (X_l,x_l)$ is (in the isometry class of)
$(X,x)$.  For the needs of the present paper it is sufficient to work with GH-limits, thus readers not familiar with ultralimits may always choose an appropriate subsequence and consider the corresponding GH-limit.

Let $f_l:(X_l,x_l)\to (Y_l,y_l)$ be $C$-Lipschitz and $C$-open maps and assume that the spaces $X_l$ are complete.
Then the ultralimit    $f=\lim _{\omega} f_l :(X,x)\to (Y,y)$ is $C$-Lipschitz and $C$-open. Moreover, for $\Pi_l:=f_l^{-1} (y_l)$
the ultralimit  $\lim_{\omega } \Pi _l  \subset (X,x)$ coincides with the fiber $\Pi := f^{-1} (y)$.


\section{Spaces with an upper curvature bound} \label{sec:prelcurv}
\subsection{Definitions and notations}
For $\kappa \in \R $, let $R_{\kappa}  \in (0,\infty] $ be the diameter of the  complete, simply connected  surface $M^2_{\kappa}$
of constant curvature $\kappa$. A complete  metric space is called \emph{$\CAT(\kappa)$}
if any pair of its points with distance $<R_{\kappa}$ is connected by a geodesic and if
 all triangles
with perimeter $< 2R_{\kappa}$
are not thicker than
 the \emph{comparison triangle} in $M_{\kappa} ^2$.
A  metric space
is called a \emph{space with  curvature bounded above by  $\kappa$}
if   any point has a  $\CAT(\kappa)$ neighborhood.
We refer to \cite{BH}, \cite{BBI01}, \cite{ballmann} for basic facts about such spaces.

Any $\CAT(\kappa )$ space is $\CAT(\kappa ')$ for $\kappa ' \geq \kappa$. By rescaling
we may always assume that the curvature bound $\kappa $ equals $1$. Then $R_{\kappa} =\pi$.

For any $\CAT(\kappa )$ space  $X$,
the angle between each pair of geodesics starting at the same point
is well defined. The \emph{space of directions}
 $\Sigma _x = \Sigma _xX$ at each point $x$, which is the completion of the set of geodesic directions equipped with the angle metric,
is a $\CAT(1)$ space.
The Euclidean cone
over $\Sigma _x$ is a $\CAT(0)$ space.
It is denoted by $T_x=T_xX$ and called
the \emph{tangent space} at $x$ of $X$.
The element $w$ in $T_x$
will be written
as $w = tv = (t,v) \in T_x = [0,\infty) \times \Sigma _x / \{0\} \times \Sigma _x$,
and its \emph{norm} is defined as
$|w| = |tv| := t$.

 Let $x,y,z$ be three points at pairwise distance $< R_{\kappa}$ in a $CAT(\kappa )$ space $X$.
Whenever $x\neq y$, the geodesic between $x$ and $y$ is unique and will be denoted
by $xy$.  Its starting direction in $\Sigma _x$ will be denoted by
$(xy)'$ if no confusion is possible. If $y,z \neq x$ the angle at $x$ between $xy$ and $xz$, hence the distance in $\Sigma _x$
between  $(xy)'$  and $(xz)'$,
will be denoted by $\angle yxz$.

For $r < R <  R_{\kappa} /2$
we consider the \emph{contraction map}
$c_{R,r} \colon B_R(x) \to B_r(x)$ centered at $x$,
that sends the  point $y$
to the point $\gamma (\frac r R \cdot d(x,y))$,
where
$\gamma$ is the unique geodesic from $x$ to $y$.
Due to  the $\CAT(\kappa )$ property,
 the map $c_{R,r}$ is $(2\cdot \frac r R)$-Lipschitz,
  compare  \cite[Section 2.1]{N2}, for the optimal estimate.

We define
the \emph{logarithmic map}
$\log_x \colon B_{\frac 1 2 R_ \kappa}(x) \to T_x$  by  $\log_x (x)=0$  and by sending any $y\neq x$
to $tv \in T_x$, with $t=d(x,y)$ and $v=(xy)'$. The $\CAT(\kappa )$ property implies that  $\log_x$ is $2$-Lipschitz.


\subsection{Basic topological properties}
On spaces with an upper curvature bound,
there is a notion of \emph{geometric dimension} invented by Kleiner \cite{Kleiner}.
The geometric dimension satisfies $\dim X = 1 + \sup_{x \in X} \dim  \Sigma _x$.
The geometric dimension
$\dim X$
is equal to the topological dimension if  $X$ is separable \cite{Kleiner}.

Convexity of all small balls in spaces with upper curvature bounds imply that
any space $X$ with an upper curvature bound is an \emph{absolute neighborhood retract},
see \cite{Ontaneda}, \cite{Linus}.
In particular, each open subset of $X$ is homotopy
equivalent to a simplicial complex.

For any $\CAT(\kappa )$ space $X$ the  map  $\log_x \colon (B_{\frac 1 2 R_{\kappa}}(x) \setminus \{ x \}) \to T_x \setminus \{ 0 \}$ is a homotopy equivalence,
\cite{Linus}.

\subsection{Convergence and semicontinuity} \label{subsec:semi}
Let  $(X_i,x_i)$ be a sequence of  pointed  $\CAT(\kappa_i)$ spaces with $\lim _{i\to \infty} \kappa _i =\kappa$.
Then $(X,x)=\lim _{\omega} (X_i,x_i)$ is $\CAT(\kappa)$, \cite{BH}.
 Moreover,
$\lim _{\omega} \dim (X_i,x_i) \geq \dim (X,x)$,   \cite[Lemma 11.1]{Lbuild}, thus, the geometric dimension does not increase under convergence.

Let  $y_i,z_i \in X_i$ be points such that  $\epsilon \leq d(x_i,y_i),d(x_i,z_i) \leq \frac {R_{\kappa}} 2$,
 for some $\epsilon >0$.
The points $y=\lim _{\omega} (y_i) \, ,z=\lim _{\omega} (z_i)$  in $X$ and  the angles
 $\angle y_ix_iz_i$ and $\angle yxz$ are well-defined. In this situation, we have
the following semicontinuity of angles $ \lim_{\omega} \angle y_ix_iz_i  \leq \angle yxz$,    see  \cite{BH}, \cite[p.748]{Lbuild}.

\section{Geodesic extension property}  \label{sec:extend}
\subsection{Definition} \label{subsec:def}
Let $X$ be a space with curvature   bounded above by $\kappa$.  We  call $X$   \emph{locally geodesically complete}
 if  any local
geodesic $\gamma :[a,b]\to X$, for any $a<b$, extends as a local geodesic to a  larger interval $[a-\epsilon, b+\epsilon]$.
 If any local geodesic in $X$ can be extended  as a local geodesic to $\R$  then $X$ is called
\emph{geodesically complete}.

 In \cite{BH} local geodesic completeness is called the \emph{geodesic extension property}.

 For any local  geodesic $\gamma :[a,b] \to X$  in a space $X$ with curvature $\leq \kappa$, we can use Zorn's lemma to find a maximal extension of $\gamma$ to a local geodesic $\gamma :I\to X$
defined on an interval $I\subset \R$.  If $X$ is locally geodesically complete, then such a maximal interval $I$ is open in $\R$.
Assume that $t=\sup (I )$ is finite. For any  $t_i\in I$ converging to  $t$, the sequence $\gamma (t_i)$ is a Cauchy sequence in $X$.
If $\gamma (t_i)$ converge to a point $x$ in $X$ then the final part  $\gamma :[t-\epsilon, t)  \to X $ is contained in a $\CAT(\kappa )$ neighborhood $U$ of $x$. Since local geodesics of length $\leq R_{\kappa} $ in $U$ are geodesics,  the unique extension of $\gamma$ by $\gamma  (t)=x$ is a geodesic
on $\gamma :[t-\epsilon ,t] \to X$. But then, contrary to our assumption, $I$ is not a maximal interval of definition of $\gamma$. Thus we have shown
that $\gamma (t_i)$ cannot converge in $X$. From this we conclude:

\begin{lem} \label{sizeext}
Let $X$ be a locally geodesically complete space with an upper curvature bound. Let the closed ball $\bar B_r (x)$  be complete. Then any local geodesic $\gamma$ in $X$  with $\gamma (0)=x$ can be extended to a local geodesic $\gamma :(-t^-,t^+) \to X$ with $t^{\pm} >r$. %
\end{lem}

\begin{proof}
Extend $\gamma$ to a maximal interval of definition $I=(-t^-,t^+)$.  If $t^+ \leq r$ then $\gamma ([0,t^+)) \subset \bar B_r (x)$. Thus $\lim _{t_i \to t^+} \gamma (t_i)$ exists in $\bar B_r (x)$ in contradiction to the observation preceding the lemma. Thus, $t^+ > r$. Similarly,
$t^- >r$.
\end{proof}
 In particular,  a complete metric space with an upper curvature bound is geodesically complete if
 it is  locally geodesically complete.

Let the space $X$ with curvature at most $\kappa$ be locally geodesically complete.
Let $x\in X$ be arbitrary. Then
some closed ball $K= \bar B_{2r} (x)$ with $4r<R_{\kappa}$ is  $\CAT(\kappa )$.  Note that any geodesic in $K$ is uniquely determined by its endpoints and any local geodesic in $K$ is a geodesic.  Due to Lemma \ref{sizeext}, for any $y\in B_r(x)$  any geodesic $\gamma$ starting in $y$ can be extended  inside $K$ to a geodesic $\gamma :[-r,r] \to X$.

\subsection{Examples} \label{subsec:ex}
The following example shows that (local) geodesic completeness  without further compactness assumption  is not of much use.
\begin{exmp}
Starting with any $\CAT(\kappa )$ space $X$ we  glue to all points $x\in X$ a line $\R =\R _x$.   The arising "hairy" space $\hat X$ is still $\CAT(\kappa )$,  geodesically complete and contains $X$ as a convex subset.
\end{exmp}

Let $X$ be a Euclidean simplicial complex with a finite number of isometry classes of simplices
 and curvature at most $0$. Then $X$ is locally geodesically complete if and only if
any face of  any maximal simplex is a face of at least one other simplex.

 For any $\CAT(1)$ space $\Sigma$, the Euclidean cone $C\Sigma$ over $\Sigma$ is  geodesically complete
 if and only if $\Sigma $ is   geodesically complete and not a singleton.
 The direct product of two $\CAT(0)$ spaces  is geodesically complete
 if and only if  both factors are geodesically complete.

  There is a simple  topological condition implying local geodesic completeness, cf. \cite[Theorem 1.5]{LS}.
  Namely, if $X$ is a space with an upper curvature bound
  and if at all points $x\in X$ the local homology $H_{\ast} (X,X\setminus \{ x \})$  does not vanish then $X$ is locally geodesically complete.
  In particular, any space with an upper curvature bound which is a (homology) manifold  is locally geodesically complete.

Geodesic completeness  is preserved under gluings: Let $X_1,X_2$ be  two spaces of curvature $\leq \kappa$ and let $A_i\subset X_i$ be  locally convex and  are isometric to each other.  The space $X$ which arises from gluing of $X_1$ and $X_2$ along $A_i$ has curvature $\leq \kappa$, by a theorem of Reshetnyak. It is a direct consequence of the structure of  geodesics in $X$, that if $X_1$ and $X_2$ are (locally) geodesically complete  then so is $X$.

Finally, geodesic completeness is preserved under ultralimits:

\begin{exmp} \label{ex:converge}
Let $(X_i,x_i)$ be locally geodesically complete spaces with curvature $\leq \kappa _i$. Assume that the balls $\bar B_{r_i} (x_i) \subset X_i$ are $\CAT (\kappa _i)$, with $2r_i \leq  R_{\kappa_i}$. Assume, finally, that $\lim_{i\to \infty} (\kappa _i)=\kappa$ and $\lim_{i\to \infty } r_i>r>0$.  Consider the ultralimit $(X,x)=\lim _{\omega} (X_i,x_i)$.  Then the closed ball $\bar B_r(x) \subset X$ is $\CAT(\kappa )$ as an ultralimit of $\CAT (\kappa _i)$ spaces.  We claim, that  the open ball $B_r(x)$ is locally geodesically complete. Indeed, any geodesic $\gamma$ in $B_r (x)$ is an ultralimit of the corresponding geodesics in $B_{r} (x_i)$. Since the latter admit extensions  of a uniform size to longer geodesics,  we obtain an extension of $\gamma$ as the corresponding ultralimit.
\end{exmp}


\section{GCBA} \label{sec:setting} \label{sec:gcba}

\subsection{GCBA spaces  and their tiny balls} \label{subsec:size}
Now we turn to the main subject of this paper, the structure of locally compact, locally geodesically complete, separable spaces with upper curvature bounds. As in the introduction, we will denote such spaces as GCBA. Then any open subset of a GCBA space  is GCBA  as well.

   We call  an open  ball $U=B_{r_0} (x_0)$ in a GCBA
 space  $X$  of curvature $\leq \kappa$
  a \emph{tiny ball} if
  the following holds true. The radius $r_0$ of  $U$ is at most  $\min \{1 , \frac 1 {100} \cdot R_{\kappa} \}$ and the closed ball  $\tilde U=\bar B_{10\cdot r_0} (x_0)$ with the same center and radius $10\cdot r_0$ is compact.

 As  seen at the end of Subsection \ref{subsec:def},  any geodesic $\gamma$ with $\gamma (0) \in U$ can be extended  to  a geodesic $\gamma :[-9 \cdot r_0, 9\cdot r_0] \to \tilde U\subset X$. For any ball $B_r(x)$ contained in $\tilde U$ and any $r'<r$, the contraction map
 $c_{r,r'} :B_r(x)\to B_{r'} (x)$ is surjective.

Any point in $X$ is contained in a tiny ball.
  Since $X$ is separable, we can write it as a countable union of  tiny balls.
   Any relatively compact subset of $X$ is covered by  finitely many  tiny balls.  All theorems from the introduction will follow once we prove them for all tiny balls in  $X$.

\subsection{Doubling property} \label{subsec:doubling} Tiny balls turn out to be doubling.

\begin{prop} \label{prop:doubl}
Let  $U=B_{r_0} (x_0) $ be a tiny ball of radius $r_0$ in a  GCBA space.  Let $N$ denote the maximal
  number of $r_0$-separated points in  the compact ball $\tilde U=\bar B_{10\cdot r_0} (x_0)$.
 Then  $\bar B_{5\cdot r_0} (x_0)$   is   $N$-doubling.
\end{prop}
\begin{proof}
 It suffices to prove that for any $t>0$ and any $y\in  \bar B_{5\cdot r_0} (x_0)$, any $\frac t 2$-separated subset
$S$ of $\bar B_t(y)$ has at most $N$ elements.

The statement is clear  for $t\geq 2r_0$, by the definition of $N$.

For $t<2r_0$,  consider the $\frac {t} {2\cdot r_0} $-Lipschitz  map $c_{4\cdot r_0 ,t} \colon B_{4\cdot r_0} (y) \to B_t(y)$.
The map is surjective, since $B_{4\cdot r_0} (y) $ is contained in $\tilde U$.
Hence, taking arbitrary preimages of points in $S$ under this contraction map, we obtain an $ r_0$-separated subset of $B_{4\cdot r_0} (y) $ with as many elements as  in $S$. Since $B_{8\cdot r_0} (y) \subset \tilde U$, we deduce that $S$ has at most  $N$ elements.
\end{proof}

\begin{defn} \label{defn: doubl}
 For a tiny ball $U=B_{r_0} (x_0)$ of a GCBA space, we say that $U$ has   \emph{capacity bounded by $N$ if  $\tilde U:=\bar B_{5\cdot r_0} (x_0)$} is $N$-doubling.
\end{defn}

Let $X$ be GCBA. Let $U\subset X$  be a tiny ball of radius $r_0$ and capacity bounded by $N$.  Then  any open ball contained  in  $U$ is a tiny ball in $X$ of  capacity bounded by $N$.
 Moreover,  for any point $x\in U$,
the ball $B_s (x)$ is a  tiny ball of capacity bounded by $N$, for any $s \leq \frac {r_0} 2$.   Finally, for every $s \leq \frac 1 {r_0}$ the rescaled space $s\cdot U$ is a tiny ball in
 the GCBA space  $s\cdot X$ with capacity bounded by the same $N$.

\subsection{Distance  maps and a biLipschitz embedding} \label{subsec:bilip}
Let  $U = B_{r_0} (x_0) \subset \tilde U :=\bar B_{5\cdot r_0} (x_0)$ be a tiny ball of radius $r_0$ and capacity bounded by $N$ as above.

For  $p \in \tilde U$
we denote by $d_p :\tilde U\to \R$
the \emph{distance function} $d_p(x) = d(p,x)$.
The function $d_p$ is $1$-Lipschitz and convex on $\tilde U$.
For any $m$-tuple of points $(p_1,...,p_m)$ in $\tilde U$ the \emph{distance map} defined by the $m$-tuple is the map $F:\tilde U\to \R^m$
with coordinates $f_i (x)=d_{p_i} (x)$.  	
Since any distance function $d_{p_i}$ is $1$-Lipschitz, any distance map $F:\tilde U\to \R^m$ is $\sqrt m$-Lipschitz. Moreover, if we equip $\R^m$ with the \emph{sup-norm},
then $F$ becomes a $1$-Lipschitz map $F:\tilde U\to \R^m _{\infty}$.


 Let  $\gamma :[a,b] \to  \tilde U$  be a geodesic starting at $x=\gamma (a)$. Let $p\neq x$ and $f=d_p$.
The derivative of $f \circ \gamma $  at  time  $a$  is
 computed by the \emph{first variation formula} $(f\circ \gamma )' (a)=-\cos (\alpha)$, where $\alpha \in [0,\pi] $ denotes the angle between $\gamma$ and the geodesic $xp$. In particular,  $|(f\circ \gamma )' (a)| < \delta$ if $|\alpha -\frac \pi 2 | <\delta$. Moreover, $(f\circ \gamma )' (a) >1-\delta$ if
 $\alpha >\pi-\delta$ and $(f\circ \gamma )' (a) <-1+\delta$ if $\alpha <\delta$.

Denote by $\mathcal A \subset \tilde U$ the compact subset of all points $p \in \tilde U$ with $r_0 = d(p, U)$, thus a distance sphere with radius $2r_0$ around the center of $U$.
Due to the assumptions on $r_0$ and the curvature bound, for all $\delta >0$ the following holds true.
For every pair of points $p,q \in \mathcal A$ with $d(p,q) \leq \delta  \cdot r_0$ and any $x \in U$ we have $\angle pxq <\delta$.

For all $\delta >0$, we choose a maximal $\delta \cdot r_0$-separated subset $\mathcal A_{\delta} $ in $\mathcal A$. Due to the
doubling property, the number of elements in $\mathcal A_{\delta}$ is bounded by some $m=m(N,\delta)$. Now we obtain:

\begin{prop}  \label{lem:embed}
For every $\delta >0$ there exists some natural $m =m(N,\delta)$ and  $m$ points  $p_1,...,p_m \in \tilde U$ such that the corresponding distance map $F: U\to \R^m _{\infty}$
is a $(1+\delta)$-biLipschitz embedding. Here $\R^m_{\infty}$ denotes $\R^m$ with the sup-norm.
\end{prop}

\begin{proof}
Consider as above the  maximal $\delta \cdot r_0$-separated subset $\mathcal A_{\delta} =\{p_1,...,p_m \}$ in  the distance sphere $\mathcal A$
and note, that $m$ is bounded in terms of $N$ and $\delta$.

 Consider the   corresponding  distance map $F:  U\to \R^m _{\infty}$.
As all distance maps,  $F:  U\to \R^m _{\infty}$ is $1$-Lipschitz.

Given arbitrary $x,y\in   U$, we extend $xy$ beyond $y$ to a point $q \in \mathcal A$.   We find some  $p_j\in \mathcal A _{\delta}$  such that $d(p_j,q)  \leq \delta \cdot r_0$
hence    $\angle p_jyq  < \delta$.   Then $\angle p_jyx > \pi -\delta$.
From the first variation formula the derivative of  the distance function $d_{p_j}$ on the geodesic $yx$ at $y$ is   at least $(1-\delta)$.

Then $d(p_j,x)-d(p_j,y) \geq (1- \delta ) \cdot d(x,y )$, due to the convexity of $d_{p_j}$.
Hence,  $|F(x)-F(y)| _{\infty}  \geq (1 -\delta)\cdot d(x,y)$. This finishes the proof.
\end{proof}

We let   $\delta =1$ in  Lemma  \ref{lem:embed} and obtain  a refinement of Proposition \ref{prop:doubl},
previously proved in \cite[Theorem 1.1]{Langplaut}.  The analogous statement in spaces with curvature bounded below is true but much deeper,
see \cite{AKP}.

\begin{cor} \label{lem:doubling}
For some $n_0=n_0(N)$,  there exists a biLipschitz embedding $F:U\to \R^{n_0}$.
The Hausdorff and  the topological dimensions of $U$ are at most $n_0$.
\end{cor}

\subsection{Almost Euclidean triangles}
The diameter of $\tilde U$ is  smaller than $\frac 1 4 R_{\kappa}$. Hence  $\angle xyz + \angle yxz  \leq \pi$ for any
 triangle $xyz$ in $\tilde U$. If $d(x,z)=d(y,z)$ then $\angle xyz <\frac \pi 2$.


 The following lemma shows that  triangles in $U$,  with one side fixed
and the other side sufficiently small, have almost Euclidean angles.

\begin{lem}
Let $x\in U$ and $p\in \tilde U$ be arbitrary.
For any $\epsilon >0$ there is some $\delta>0 $  such that
for any  $y\in B_{\delta} (x)$ we have $\angle pxy + \angle pyx > \pi -\epsilon$.
\end{lem}

 \begin{proof}
Assume the contrary and take a sequence $y_i$ converging to $x$  and satisfying
 $\angle pxy_i + \angle py _ix \leq  \pi -\epsilon$.  Extend
the geodesic $xy_i$ beyond $y_i$  up to a point $z_i$ with
$d(x,z_i) = r_0$. Choosing a subsequence we may
assume that $z_i$ converges to a point $z$. 
The semicontinuity
of angles gives us
$$\lim \angle  p x z_i = \angle pxz
\geq \limsup \angle  p y_iz_i \,.$$
This contradicts
$\angle py_ix \geq \pi - \angle py_iz_i$ and finishes the proof.
\end{proof}

\subsection{Tangent spaces and spaces of directions}
 We fix an arbitrary   $x\in U$  and claim that every  $v\in \Sigma _x$ is the starting direction of a geodesic of length $5r_0$.  Thus,
 for any   $r\leq 5r_0$, the  map
$\log_x :B_{r} (x) \to T_x$ has the  ball $B_r (0) \subset T_x$ as its image.

 Indeed,
write $v$  as a limit of starting directions $(xy_i)'$ of geodesics.
We extend $xy_i$ to geodesics $xz_i$ of length $5r_0 $ and find a subsequence converging to a geodesic $xz$ with starting direction $v$.

The restriction of $\log _x$ to small balls is an almost isometry:
\begin{lem} \label{lem:tangent}
For any $\epsilon >0$ there is some  $\delta >0$ (depending on the point $x$), such that for all
$r  <\delta$ and all $y_1,y_2 \in B_r (x)$ we have
\begin{equation} \label{above}
|d(y_1,y_2) - d(\log_x(y_1), \log_x (y_2))| \leq \epsilon \cdot r \, .
\end{equation}
\end{lem}

\begin{proof}
We find some finite $ \epsilon \cdot r_0 $-dense subset  $\{ p_1,.....,p_m \}$ in   $U=B_{r_0} (x)$.
Then the union of geodesics $xp_i$ is $2\epsilon  \cdot  r$ dense in $B_r(x)$,  for any $r<r_0$.

By the definition of angles, we find a sufficiently small
$\delta >0$ such that \eqref{above} holds true for all $y_1,y_2$ which lie on the union of the finitely many geodesics $xp_i$.
Since the logarithmic map is $2$-Lipschitz, we conclude \eqref{above} with $\epsilon$ replaced by $9\epsilon$, for arbitrary
$y_1,y_2\in B_r(x)$.
 \end{proof}

Thus, the logarithmic map provides an almost isometry between rescaled  small balls in $X$ and corresponding balls in the tangent space.
From the definition of GH-convergence this  implies:
\begin{cor} \label{cor:converge}
For any sequence $t_i\to 0$ the rescaled spaces $(\frac 1 {t_i} \bar U,x)$ converge in the pointed GH-topology to the tangent space $(T_x,0)$.
\end{cor}

From the stability of the geodesic extension property discussed in  Example  \ref{ex:converge}  and the doubling property of  $U$ we see:
\begin{cor} \label{compcone}
For any $x\in U$ the tangent space $T_x$ is an $N$-doubling,  geodesically complete $\CAT(0)$ space.
\end{cor}

We derive:
\begin{cor}  \label{complink}
For any  $x\in U$ the space of directions $\Sigma_x$ is a compact, geodesically complete $\CAT(1) $ space.
$\Sigma _x$  is $N_1$-doubling with $N_1$
depending on $N$.
If $U$ is not a singleton then $\Sigma _x$ has diameter $\pi$.
\end{cor}

\begin{proof}
If $U$ is a singleton then $\Sigma _x$ is empty. Otherwise,  there exists at least one geodesic passing through $x$, hence $\Sigma _x$ is not empty and has diameter at least $\pi$. By the definition of the angle metric, the diameter of $\Sigma _x$ cannot be larger than  $\pi$.  The doubling property  follows from Corollary \ref{compcone}, since $T_x$ is the Euclidean cone
over $\Sigma _x$  and the embedding of $\Sigma _x$ into $T_x$ is $2$-biLipschitz.
\end{proof}

\subsection{Precompactness and setting for convergence}		A bound on the  amount and capacities of tiny balls in a covering is equivalent to precompactness in the GH-topology, once the bounds on the curvature and injectivity radius are fixed:

\begin{prop} \label{prop:precomp}
Let  $\kappa, t  >0$ be fixed.
Let $X_l$ be GCBA  and let $K_l \subset X_l$ be  compact and connected.
Assume  that for any $x_l \in K_l$ the ball $\bar B_t (x_l)$ in $X_l$ is a compact $\CAT(\kappa )$  space.
Then the following are equivalent:
\begin{enumerate}
\item There exists  $r>0$ such that closed tubular neighborhoods $\bar B_r (K_l)$ are  uniformly compact, i.e., each one is compact and they constitute a precompact set in the Gromov--Hausdorff topology.

\item There are  $r,N>0$, such that the closed tubular neighborhoods $\bar B _r(K_l)$ are compact, have diameter   at most  $ N$  and are  $N$-doubling.

\item There are some $r,N>0$ and  a covering of   $\bar B_r (K_l)$ by at most $N$ tiny balls of capacity  bounded by $N$.
\end{enumerate}
\end{prop}
	
\begin{proof}
The implication (2) to (1) is clear.

Under the assumptions of (3), $\bar B_r (K_l)$ is $N^3$-doubling by the definition of the bound on the capacity. Moreover,
the diameter of $\bar B_r (K)$ can be at most $2\cdot N$, since the diameter of any tiny ball is at most $2$ and $\bar B_r(K)$ is connected, at least for all $r\leq t$.  Thus (3) implies (2).

Assume (1). We find  some $s<\frac {r} {40}$ such that for any $x\in K_l$ the open ball $B_{2s} (x)$ is tiny in $X_l$.
 By the assumption of uniform compactness, there is some $N>0$ such that  the maximal $s $-separated subset in $\bar B_r (K_l)$ has at most $N$ elements.  Hence, we can cover $\bar B_r(K_l)$ by at most $N$ open balls of radius $2s$ and each of these tiny  balls has capacity at most $N$, due to Proposition \ref{prop:doubl}.  This implies (3).
		\end{proof}
		
As a consequence of Example \ref{ex:converge}, we see:

\begin{cor} \label{cor:precomp}
Under the equivalent conditions of Proposition \ref{prop:precomp},  the compact subsets $K_l\subset X_l$ converge, upon choosing a subsequence, in the GH-topology to a compact subset $K$ of a GCBA space  $X$. There is some $s>0$ such that the compact  neighborhoods  $\bar B_{10\cdot s}(K_l) \subset X_l$ converge in the GH-topology to the compact  neighborhood
$\bar B_{10 \cdot s} (K) \subset X$.
 \end{cor}	

We can choose $s$ in  Corollary \ref{cor:precomp}  to be much smaller than $1$ and  the injectivity radius $t$.  Then all balls with radius
$s$ centered in $K_l$ or in $K$ are tiny balls in $X_l$ and $X$ respectively.  Therefore,
	in all local questions concerning convergence, we can restrict ourselves to a convergence   of tiny balls in some GCBA spaces to a tiny ball in some other GCBA space, as described in the following.

\begin{defn} \label{defn:standard}
As the \emph{standard setting for convergence} we will denote the following situation.  The sequence $U_l\subset \tilde U_l$ of tiny balls in GCBA spaces
$X_l$ have the same radius $r_0$ and the same bound on the capacity $N$.   The sequence $\tilde U_l$ converges in the GH-topology to a compact ball
$\tilde U$ of radius $10\cdot r_0$ in a GCBA space $X$.  The closures $\bar U_l$ converge to the closure $\bar U$ of a tiny ball $U\subset \tilde U$ of radius $r_0$ in  $X$.
\end{defn}

\subsection{Semicontinuity of tangent spaces}
For GCBA spaces, semicontinuity of angles discussed in Subsection \ref{subsec:semi} has the following nice formulation.

\begin{lem} \label{lem: semiproj}
Under the standard setting of the convergence as in Definition \ref{defn:standard}, let $x_l\in U_l \subset \tilde U_l$ converge to $x\in U \subset \tilde U$.
Then the sequence of the spaces of directions $\Sigma _{x_l} U_l$ is precompact in the GH-topology. For every limit space $\Sigma '$ of this sequence there exists a surjective $1$-Lipschitz map $P:\Sigma _xU \to \Sigma '$.
\end{lem}

\begin{proof}
 Corollary \ref{complink} implies that the sequence $\Sigma _{x_l} U_l$ is uniformly doubling, hence precompact.

 In order to prove the second statement, we may replace our sequence $U_l$ by a subsequence and assume that $\Sigma _{x_l} $ converge to $\Sigma'$.

 For any direction $v\in \Sigma _x U$ we take a point $y\in \tilde U$ with $(xy)'=v$ and $d(x,y)=r_0$.
  Consider a sequence $y_l\in \tilde U_l$ converging to $y$ and put $v_l:= (x_ly_l)' \in \Sigma _{x_l} U_l$. Then we choose the limit point $w =\lim _{\omega} (v_l) \in \Sigma'$ of the sequence $(v_l)$  and set $P(v):=w$.

 The semicontinuity of angles discussed in Subsection \ref{subsec:semi}  is exactly the statement that the map $P$ is $1$-Lipschitz.
 The surjectivity  of $P$ follows from the construction and the  fact that any direction $w\in \Sigma'$ is a limit direction of some directions $v_l \in \Sigma _{x_l} U_l$, which  are starting directions of  geodesics of
 uniform length $r_0$ in $U_l$.
\end{proof}

\section{Almost suspensions} \label{sec:spherical}

\subsection{Spherical  and almost spherical points}
In this section let $\Sigma$ be  a compact, geodesically complete  $\CAT(1)$ space  with diameter $\pi$.
Note, that any space of directions $\Sigma_x$ of any GCBA space $X$ satisfies this assumption by  Corollary \ref{complink}.

\begin{defn}
Let $\Sigma$ be  a compact $\CAT(1)$ space which is GCBA and has   diameter $\pi$.   For $v\in \Sigma$,  an \emph{antipode}
of $v$ is a point $\bar v$ with $d(v,\bar v)=\pi$.
A point $v\in \Sigma $ is called \emph{spherical} if it has only one antipode.
\end{defn}

Consider the subset $\Sigma ^0$ of all spherical points  $v\in \Sigma$. Then
$\Sigma ^0$ is a convex subset isometric to some  unit sphere $\Sph ^k$ and $\Sigma $ is a spherical join
 $\Sigma =\Sigma ^0 \ast \Sigma '$, see, for instance,  \cite[Corollary 4.4]{Lbuild}.
The Euclidean cone $C\Sigma$ has an $\R^{k}$-factor if and only if $\Sigma $ is decomposable as a spherical
join of $\Sph ^{k-1}$ and another space. Moreover, the maximal Euclidean factor is   $C\Sigma ^0 \subset C\Sigma$.

\begin{defn} \label{def:deltasph}
Let $\Sigma $ be as above    and let  $\delta >0$ be arbitrary. We call a point $v\in \Sigma$ a \emph{$\delta$-spherical point}, if there exists some $\bar v\in \Sigma$ such that
for any $w\in \Sigma$
\begin{equation} \label{eq:opp}
d(v,w)+d(w,\bar v) <\pi +\delta\, .
\end{equation}
Moreover, we say that  $v$ and $\bar v$ are  \emph{opposite} $\delta$-spherical points.
\end{defn}

The triangle inequality and extendability  of geodesics to length $\pi$   directly imply:

\begin{lem}\label{lem: antisph}
Let $\Sigma$ be as above.  The points  $v,\bar v \in \Sigma$ are opposite $\delta$-spherical  points if and only if
$d(\bar v,w)<\delta$ for any antipode $w$ of $v$. In particular, in this case $d(v,\bar v)>\pi-\delta$ and the set of all antipodes of $v$
 has diameter less than $2\delta$.  Finally, for every antipode $v'$ of $v$, the pair $(v,v')$ are opposite $2\delta$-spherical points.
\end{lem}

\subsection{Tuples of  $\delta$-spherical points}
We define special positions of pairs of almost spherical points:

\begin{defn} \label{def:deltaort}
Let $\Sigma $ be as above. Let $(v_1,...,v_k)$ be a  $k$-tuple of
points in $\Sigma$. We say that $(v_i)$ is a \emph{$\delta$-spherical $k$-tuple}
if there exists another $k$-tuple $(\bar v_i)$ in $\Sigma$ with the following two properties.

(1) For $1\leq i \leq k   \;  $  :   $ \;   v_i$ and $\bar v_i$  are opposite $\delta$-spherical points.

(2) For  $1\leq i\neq j \leq k  \; $   : $\; d(v_i,\bar v_j)   <\frac  \pi 2 + \delta \; ; \;  d(v_i, v_j)< \frac \pi 2 + \delta \;  ; \;   d(\bar v_i,\bar v_j)< \frac \pi 2 + \delta \,.$

Moreover,  $(\bar v_i)$ and $(v_i)$ are called \emph{opposite $\delta$-spherical $k$-tuples}.
\end{defn}

From    Lemma \ref{lem: antisph} and the triangle inequality  we deduce:
\begin{cor} \label{cor:complement}
Let $\Sigma$ be as above. Let $v_1,...,v_k \in \Sigma$ be $\delta$-spherical points.  If $(v_1,...,v_k)$ is a $\delta$-spherical $k$-tuple then, for all $i\neq j$,
$$\frac \pi 2 -2\delta < d(v_i,v_j) <  \frac \pi 2 + \delta \,.$$
 On the other hand,  assume that, for all $i\neq j$,
$$\frac \pi 2 -\delta < d(v_i,v_j) <  \frac \pi 2 + \delta \, .$$
Then,  for arbitrary antipodes $\bar v_i$ of  $v_i$,  the tuples $(v_i)$ and $(\bar v_i)$ are opposite
  $2\delta$-spherical $k$-tuples.
\end{cor}

It is important to notice that all definitions above only use upper bounds on distances. Thus, due to the semicontinuity of angles,
they are suitable to provide open conditions on spaces of directions.

\subsection{Connection with GH-topology}  The existence of almost spherical $k$-tuples is equivalent to a small distance from   a $k$-fold suspension:
\begin{prop} \label{lem:equiv}
Let $\mathcal C$ be a compact set in the GH-topology  of (isometry classes of)
compact, geodesically complete  $\CAT(1)$ spaces with diameter $\pi$.  Let $k$ be a natural number. The following  are equivalent for any sequence $\Sigma _l$  in $\mathcal C$.
\begin{enumerate}
\item Any accumulation point $\Sigma \in \mathcal C $ of the  sequence   $\Sigma _l$ is isometric to a $k$-fold suspension  $\Sph^{k-1} \ast \Sigma '$, with possibly empty $\Sigma'$.

\item For any $\delta >0$ and all sufficiently large $l$, the space $\Sigma _l$ admits a $\delta$-spherical $k$-tuple.
\end{enumerate}
\end{prop}

\begin{proof}
Choosing a subsequence we may restrict ourselves to the case that $\Sigma _l$ converges to a space $\Sigma$.

A sequence of $\delta _l$-spherical $k$-tuples in $\Sigma _l$ with $\delta_l \to 0$ converges to a  $k$-tuple of spherical points  in $\Sigma$
with pairwise distance $\frac \pi 2$. This spherical $k$-tuple determines a splitting $\Sigma =\Sph^{k-1} \ast \Sigma '$, hence (2) implies (1).

On the other hand, if $\Sigma =\Sph^{k-1} \ast \Sigma '$, we choose the standard coordinate directions $e_1,...,e_k\in \Sph ^{k-1} \subset \Sigma$ and consider in $\Sigma _l$
tuples of points converging to the $k$-tuple $(e_i)$.   These $k$-tuples satisfy the condition of (2), finishing the proof.
\end{proof}



\section{Strainers} \label{sec:strainer}
\subsection{Strained points}
The following definition, translated from \cite{BGP} to our setting, is central for all subsequent considerations.

\begin{defn} \label{def:deltastr}
Let $X$ be GCBA, $k$ an integer and  $ \delta >0$. A  point $x\in X$ is   \emph{$(k,\delta)$-strained} if  the space of directions $\Sigma_x$ contains some $\delta$-spherical $k$-tuple.
\end{defn}

As  in the introduction, we denote by $X_{k,\delta}$ the set of $(k,\delta)$-strained points in $X$.  We have
$ X_{k,\delta}\subset X_{k-1,\delta} ....\subset X_{1,\delta} \subset X_{0,\delta} =X$.
Due to Proposition \ref{lem:equiv},    $X_{k,0}=\bigcap _{\delta >0}  X_{k,\delta}$ is
the set  of all points $x\in X$, for which the tangent space $T_xX$ splits off the Euclidean space $\R^{k}$ as a direct factor.

\subsection{Strainers}
As in Section \ref{sec:setting} we  fix a tiny ball $U\subset \tilde U \subset X$.
\begin{defn}\label{defn: strainer}
 Let $x \in U$ be a point and let $\delta >0$ be arbitrary.
A  $k$-tuple  of points $p_i \in \tilde U \setminus \{x\}$ is  a \emph{$(k,\delta)$-strainer
at $x$}
if the $k$-tuple  of the starting directions $((xp_i)')$ is
$\delta$-spherical in $\Sigma_x$.

Two $(k,\delta)$-strainers  $(p_i)$ and $(q_i)$  at $x$
are  \emph{opposite} if the $\delta$-spherical  $k$-tuples $((xp_i)')$ and $((xq_i)')$ are opposite in $\Sigma _x$.

For a set $V\subset U$,  a  $k$-tuple  $(p_i)$ of points in $\tilde U$ is a \emph{$(k,\delta)$-strainer in $V$} if $(p_i)$ is a $(k,\delta)$-strainer at all $x\in V$. If  $(k,\delta)$-strainers $(p_i)$ and $(q_i)$ are opposite at all points  $x\in V$, we say that $(p_i)$ and $(q_i)$ are
\emph{opposite $(k,\delta)$-strainers} in $V$.
\end{defn}

A point $p$ is a $(1,\delta)$-strainer at $x$ if and only if  there is some $v\in \Sigma _x$ such that
any continuation of $px$ beyond $x$ as a geodesic encloses an angle smaller than $\delta$ with $v$.
%
The  following observation is the most fundamental source of strainers.

\begin{prop}\label{lem: first}
For any $\delta > 0$ and $p \in U$,
there is a neighborhood $O$ of $p$ such that
the point $p$ is a $(1,\delta)$-strainer in $O\setminus \{ p \}$.
\end{prop}

\begin{proof}
Otherwise we find points $x_i \neq  p$ arbitrary close to $p$ such that $(x_ip)'$ is not $\delta$-spherical. Set $s_i=d(x_i,p)$ and extend
$px_i$ by different geodesics to points $y_i,z_i$ with $d(y_i,x_i)=d(z_i,x_i)=s_i$ and $\angle y_ix_iz_i \geq \delta$.

By construction, $d(y_i,p)=d(z_i,p)=2\cdot  s_i$ and $\log _p (y_i)=\log _p (z_i)$. On the other hand $d(y_i,z_i) \geq \rho \cdot  s_i$, where $\rho >0$ depends only   on $\delta$ and the curvature bound $\kappa $.  For $s_i \to 0$, this contradicts Lemma \ref{lem:tangent}.
\end{proof}

\begin{rem}
An observation similar to Proposition \ref{lem: first}
can be found  in \cite{Otsu}.
\end{rem}

 For any $\delta >\pi$ and  any $x\in U$, any $k$-tuple $(p_i)$ of points in $\tilde U\setminus \{x  \}$ is
 a $(k,\delta )$-strainer at $x$.  On the other hand, we have:

  \begin{lem} \label{lem:uniformbound}
  There exists a number $k_0(N)$  with the following property. For   any tiny ball $U$ of capacity
  bounded by $N$ and any  $1\geq \delta >0$, there do not exist $(k,\delta)$-strained points in $U$ with $k>k_0$.
   \end{lem}
 \begin{proof}
 Let $x$ be  a $(k,\delta )$-strained point in a tiny ball $U$ of capacity bounded by $N$.
By definition, we  find in $\Sigma _x$ a $(\frac \pi 2 -\delta) $-separated subset with $k$ points.
From the bound on the doubling constant (Corollary
\ref{complink}) and the assumption $\frac \pi 2 -\delta >\frac 1 2 >0$,
we deduce  that $k$ is bounded from above in terms of $N$.
 \end{proof}

\subsection{Almost Euclidean triangles}  The existence of strainers implies the existence of many almost Euclidean triangles. We will only  use the following:

\begin{lem}\label{lem: almperp}
Let $p,q \in \tilde U$ be opposite $(1,\delta)$-strainers at points $x\neq y$ in a tiny ball $U$.
Then the following hold true.
\begin{enumerate}
\item  $\pi -2\cdot \delta < \angle pxy +\angle pyx <\pi  \,.$
\item If $d(p,x) = d(p,y)$ then $\frac \pi 2 - 2\cdot \delta  < \angle pxy < \frac  \pi 2$.
\end{enumerate}
\end{lem}

\begin{proof}
From  the  assumption on the upper curvature bound and diameter   of $\tilde U$
we deduce the right hand side inequalities in (1)  and in (2). By the same reason
\begin{equation} \label{eq:first}
\angle qxy + \angle qyx <\pi.
\end{equation}

On the other hand,
by the definition of opposite strainers
we have
$$\angle pxy + \angle qxy  > \pi - \delta \; \;  \text{and} \; \; \angle pyx + \angle qyx  > \pi - \delta \,.$$
Hence the sum of these four  angles is at least $2\pi -2\delta$.
Combining with  \eqref{eq:first}  we deduce the left hand   side inequalities.
\end{proof}

\subsection{Stability of strainers}
If $(p_i)$ is a $(k,\delta)$-strainer  at  $x \in U$ and  $\hat p_i \in \tilde U\setminus \{ x \} $ is any point 
 on an extension of $xp_i$ beyond $p_i$
then $(\hat p_i)$ is still   a  $(k,\delta)$-strainer at $x$.

From   Corollary \ref{cor:complement} we obtain:

\begin{lem} \label{lem:qi}
Let $p_1,...,p_k \in \tilde U $ be  $(1,\delta)$-strainers at $x\in U$.  Let $q_i \in \tilde U$ be arbitrary points  lying on an extension of the geodesic $p_ix$ beyond $x$. If
 $|\angle p_ixp_j -\frac \pi 2 |< \delta$,  for all $i\neq j$,  then
$(p_i)$ and $(q_i)$ are opposite  $(k,2\delta)$-strainers  at $x$.
\end{lem}

The  definition of  strainers is designed  to satisfy the following openness condition:

\begin{lem} \label{lem:open}
Let  $U_l\subset \tilde U_l \subset X_l$ converge to  $U\subset \tilde U\subset X$   as in our  standard  setting
in Definition \ref{defn:standard}.
Let $(p_i)$ and $(q_i)$ be opposite $(k,\delta)$-strainers at $x\in U$.  Let, for  $i=1,..,k$, the sequences $p_i^l, q_i^l, x^l\in \tilde U_l$ converge to $p_i$, $q_i$ and $x$, respectively.

Then the $k$-tuples $(p_i^l)$ and $(q_i^l)$ in $\tilde U_l$ are opposite $(k,\delta )$-strainers at the point  $x^l$,  for all $l$ large enough.
\end{lem}

\begin{proof}
The claim is a  consequence of the semicontinuity of angles  under convergence, Subsection \ref{subsec:semi} (see also Lemma \ref{lem: semiproj}), and the definition of $\delta$-spherical $k$-tuples, which only  involves  strict  upper bounds on distances.
\end{proof}

Restricting to the case $U_l=U$, for all $l$, we see from Lemma \ref{lem:open}
\begin{cor}  \label{cor: openopp}
For  $k$-tuples   $(p_i)$ and $(q_i)$  in $\tilde U$, the set of points $x\in U$ at which $(p_i)$ and $(q_i)$ are opposite   $(k,\delta)$-strainers  is  open.

 The set of points $x\in U$ at which $(p_i)$ is a $(k,\delta)$-strainer is open.
 \end{cor}

\subsection{Straining radius} \label{subsec:special}
We will need some uniformity in the choice of opposite strainers and the diameters of strained neighborhoods.
As before, we denote by $r_0$  the radius of the tiny ball $U$.

\begin{lem} \label{lem:opp}
Let  $(p_i)$ be a $(k,\delta)$-strainer  at  $x\in U$. Then  there  is  some number
 $0<\epsilon _x  <\frac 1 2 \cdot d(x,\partial U)$  with the following  property.
 If $y\in B_{\epsilon _x} (x)$  is arbitrary and
  $q_i \in \tilde U$,  with   $d(q_i, y)=r_0$, lies on an arbitrary continuation of $p_iy$ beyond $y$,
  then  the $k$-tuples  $(q_i)$ and $(p_i)$ are opposite $(k,2\delta)$-strainers  in
   the  ball $B_{\epsilon _x} (y)$.
\end{lem}

\begin{proof}
In order to prove the  statement, we assume the contrary and  find  contradicting sequences  $y_l, z_l\to x$ and  $k$-tuples  $(q_i^l)  $. Thus, $d(y_l,q_i^l) =r_0$, the point
 $y_l$ is on the geodesic $p_i q_i^l$, and  $(p_i^l)$ and
$(q_i^l)$  are not  opposite $(k,2\delta)$-strainers at $z_l$. Taking limit points we find a $k$-tuple $(q_i) \in \tilde  U$ such that $x$ is an inner point of  the geodesic $p_iq_i$ for any $i$.

 Due to stability of strainers,  Lemma \ref{lem:open},  $(q_i)$ and $(p_i)$ cannot be  opposite $(k,2\delta)$-strainers at $x$.
But this contradicts Lemma \ref{lem:qi}.
\end{proof}

We will call the maximal number $\epsilon _x$ as in Lemma \ref{lem:opp} above the
 \emph{straining radius at $x$}, of the $(k,\delta)$-strainer
 $(p_i)$.
By definition, $\epsilon_y \ge \epsilon_x - d(x,y)$, for all $x,y$ strained by $(p_i)$.  In particular, the  map
$x\to \epsilon _x$ is continuous.

Note finally, that the proof above literally transfers to the convergence setting from Lemma \ref{lem:open}.
Thus, the proof shows:

\begin{lem} \label{lem:strradius}
Under the  assumptions of Lemma \ref{lem:open}, let $\epsilon _x$  and $\epsilon _{x_l}$ be the straining radius of $(p_i)$ and $(p_i^l)$ at $x$ and $x_l$, respectively.
Then $\liminf _{l\to \infty}  \epsilon _{x_l} \geq  \epsilon _{x}$.
\end{lem}

\section{Strainer maps} \label{sec:strmaps}

\subsection{Differentials of distance maps and a criterion for openness}
 Let $U\subset \tilde U \subset X$ be a tiny ball as in Definition \ref{defn: doubl}.
As before, denote by $d_p : \tilde U\to \R$ the distance function to the point $p \in \tilde U$.

For any point $x \in \tilde U$ we collect the directional derivatives of $f=d_p$ to a differential $D_xf :T_x\to \R$. If $x\neq p$, $v\in \Sigma _x$ and $t\geq 0$ then  the differential $D_xf (tv)$ is given by the first variation formula as
$D_xf(tv)=-t  \cdot \cos (\alpha) $, where $\alpha$ is  the distance in $\Sigma _x$ between $v$ and the starting direction  of the geodesic $xp$.

If $F=(d_{p_1},...,d_{p_k}):\tilde U\to \R^k$ is a distance map, we denote as its differential $D_xF:T_x \to \R^k$ the map whose coordinates
are the differentials of $d_{p_j}$ at $x$.

The following criterion  is essentially taken from \cite[Section 11.5]{BGP}.

\begin{lem} \label{lem:genfunction}
Set  $\rho=\frac 1 {4k}$. Let $(p_i)$ be a $k$-tuple in $\tilde U$ and   let  $f_i=d_{p_i}$ be the corresponding distance functions.
 Assume that for every $x$ in an open subset  $V$ of $U$ and
 any $1\leq i \leq k$ there are directions $v^{\pm}  _i\in \Sigma _x$ with
\begin{equation} \label{eq:opendif}
 \pm D_xf_i (v^{\pm } _i) > 1-\rho \;\;  \text{and} \;\; |D_xf_j (v^{\pm}_i) |< 2\cdot  \rho, \;  \text{for} \;  i\neq j \,.
\end{equation}
 Then the distance map $F=(f_1,...,f_k) :V\to \R^k _1$  is   $2$-open if we equip $\R^k$ with the $L^1$-norm $|(t_i) |_1=\sum_{i=1}^k |t_i|$.
\end{lem}

\begin{proof}
 Let  $x\in V$ be arbitrary and $r>0$ such that $\bar B_{2r} (x)$ is complete.  Let
 $\mathfrak t  =(t_1,...,t_k) \in \R^k$ with $s:=|\mathfrak t - F(x)|_1 < r$ be fixed.
 In order to find $y\in B_{2r} (x) \cap F^{-1} (\mathfrak t)$ we consider the function
 $h:V\to \R$ given by
 $$h(z):= |\mathfrak t - F(x)|_1 -|\mathfrak t - F(z)|_1   = s-   |\mathfrak t - F(z)|_1\, .$$
 Then $h(x)=0$ and we are looking for $y\in B_{2r} (x)$ with $h(y)=s$.

 For every $z\in V$ with $h(z)<s$, there is some $i=1,...,k$ such that
 $t_i \neq f_i (z)$.   On the geodesic $\gamma$ starting at such  $z$ in the direction $v _i ^{\pm}$ (depending on the sign
 of $t_i-f_i (z)$), the value of $|t_i-f_i (z)|$ decreases  (infinitesimally) with velocity larger than $1-\rho$, while
 the values of $|t_j-f_j (z)|$ for $j\neq i$ increase  with velocity less than $2\rho$.
 Therefore the \emph{norm of the gradient of $h$} at any  $z \in V\setminus F^{-1} (\mathfrak t)$ satisfies
 $$|\nabla _z h| := \limsup _{y\to z} \frac {h(y)-h(z)} {d(y,z)}  \geq (1-\rho) -(k-1) \cdot (2\cdot \rho) >
 1-2\cdot k  \cdot \rho  \geq \frac 1 2 \;.$$
 Due to \cite[Lemma 4.1]{Ly-open}, for any $s'<s$ we find some $z \in B_{2s} (x)$ with $h(z)=s'$. We let $s'$ go to $s$ and use compactness
 of $\bar B_{2r} (x)$  to find the desired point $y$.

 This shows $B_r (F(x))\subset F(B_{2r} (x) )$ and finishes the proof.
\end{proof}

\subsection{Strainer  maps} Let $V\subset U$ be a subset, let  $(p_i)$ be a $k$-tuple in $\tilde U$,  and let  $F=(d_{p_1},...,d_{p_k}) :V\to \R^k$  be the corresponding distance map. We  say that  $F$ is a \emph{$(k,\delta)$-strainer map} in  $V$ if $(p_i)$ is a $(k,\delta)$-strainer in $V$. In this case, we  define the   \emph{straining radius of $F$ at $x\in V$ to be} the straining radius
of the
$(k,\delta)$-strainer  $(p_i)$ at $x$.
We say that distance maps $F,G$ are \emph{opposite} \emph{$(k,\delta)$-strainer maps} in $V$ if their defining   $k$-tuples are opposite $(k,\delta)$-strainers in $V$.

Let   $F:V\to \R^k$ be a $(k,\delta)$-strainer map with coordinates $f_i=d_{p_i}$ defined on an open subset $V$ of $U$. For any  $x\in V$, we find some distance map $G=(d_{q_i})$ such that $F$ and $G$ are opposite $(k,\delta)$-strainer maps at $x$.  Choose $v_i^{\pm} \in \Sigma _x$ to be the starting directions of $xp_i$  and $xq_i$, respectively.  By the definition of strainers and the first
  variation formula, we see that
\eqref{eq:opendif} hold true at the point $x$ with $\rho$ replaced by $\delta$.  Replacing the $L^1$-norm by the Euclidean norm we get:

\begin{lem}  \label{lem:openstr}
If $\delta \leq \frac 1{4\cdot k}$ then any $(k,\delta)$-strainer map $F:V\to \R^k$ on any open non-empty  set $V\subset U$ is
$L$-Lipschitz and  $L$-open with $L=2 \sqrt k$.   In particular, the Hausdorff dimension of $V$ is at least $k$.
\end{lem}

\begin{proof}
The Lipschitz property is true for any distance map. The openness constant follows from Lemma \ref{lem:genfunction}.
The bound  on the Hausdorff dimension follows, since the image $F(V)$ is open in $\R^k$.
\end{proof}

   As in \cite[Section 11.1]{BGP},  one can  derive from  Lemma \ref{lem:openstr} that the intrinsic metric on the fibers of $F$ is locally biLipschitz equivalent to the induced metric.  Since it is not used in the sequel, we do not provide the proof.

\subsection{Convergence of maps and an improvement of constants}
We are going to prove that for small $\delta$, the constant $L$ in Lemma \ref{lem:openstr} can be chosen arbitrary close to $1$.
These results will only be used in  Subsection \ref{subsec:convregular},
 where one could rely  on results from \cite{N2} instead.
 For this reason, the argument in this subsection will be sketchy. Readers not familiar with ultralimits may restrict to
  the case of tiny balls with uniformly bounded capacities and replace ultralimits by GH-limits of a subsequence, using Corollary \ref{compcone}.

Let $F:V\to \R^k$ be a $(k,\delta)$-strainer map on an open subset $V$ of some tiny ball $U$. For $\delta \leq \frac 1 {4\cdot k}$,  the map $F$ is $L$-open and $L$-Lipschitz on $V$ with $L=2\sqrt k$. Therefore, the differential $D_xF:T_x \to \R^k$ which is a limit of  rescalings of $F$ is $L$-Lipschitz and $L$-open.

Let $F_l :V_l\to \R^k$ be  a sequence of $(k,\delta _l)$-strainer maps with $\lim _{l\to \infty} \delta _l =0$.  Let $x_l\in V_l$ be arbitrary
and consider   the sequence of differentials $D_{x_l} F_l :T_{x_l} \to \R^k$.

 As in the proof of Proposition \ref{lem:equiv}, we see that the ultralimit $T=\lim_{\omega} T_{x_l}$  is a Euclidean cone which splits as $T=\R^k \times T'$.  Moreover, the ultralimit  $P=\lim _{\omega} D_{x_l} F_l :T\to \R^k$  is just the projection of $T$ onto the direct factor $\R^k$.

 Since the maps $D_{x_l} F_l$ are all $L$-open, any fiber $P^{-1} (w)$ with $w\in \R^k$ coincides with the ultralimit of fibers $\lim _{\omega} (D_{x_l} F_l ) ^{-1} (w_l)$ for any sequence $w_l$ converging to $w$.  For any unit vector $w\in \mathbb S^{k-1}$,  the fiber  $P^{-1} (w)$ has distance $1$ to the origin of the cone $T$. Therefore,  for any $\epsilon >0$,
 the distances of infinitely many of the fibers $(D_{x_l} F_l ) ^{-1} (w_l)$ to the origin must be between $1-\epsilon$ and $1+\epsilon$.

Arguing by contradiction we conclude:

\begin{lem}
For every  $k\in \N$  and  $L>1$ there exists some $\delta=\delta (L,k) >0$ such that the following holds true.
For any $(k,\delta)$-strainer map $F$ at a point $x$ in  a
tiny ball $U$  of a  GCBA space $X$
 the differential $D_xF:T_x\to \R^k$ satisfies:
\begin{enumerate}
\item $|D_xF(v)| < L$, for any $v\in \Sigma _x \subset T_x$.
\item For any  $u\in \mathbb S^{k-1} \subset \R ^k$, there exists  $v\in T_x$ with  $D_xF(v)=u$ and $|v|<L$.
\end{enumerate}
\end{lem}


The infinitesimal characterization of $L$-Lipschitz and $L$-open  maps,  \cite[Theorem 1.2]{Ly-open},  now directly implies:
\begin{cor} \label{cor:openstr}
For any $L >1$ there is some $\delta  =\delta (L,k) >0$ such that the following holds true.
Any $(k,\delta )$-strainer map $F:V\to \R^k$  is  $L$-open and  $L$-Lipschitz,
whenever $V$ is an open  convex subset of a tiny ball of capacity bounded by $N$.
\end{cor}

\subsection{Differentials of strainer maps}
  From Lemma \ref{lem: almperp}  we  easily derive:
\begin{prop} \label{prop: differential}
Let $F: V\to \R^k$ be a $(k,\delta)$-strainer map with $\delta \leq \frac 1 {4\cdot k}$.
Let $\epsilon$ be the  straining radius of $F$  at  a point $p\in V$.   Let $\gamma :I\to B_{\epsilon} (p)$ be a geodesic
defined on an interval $I$.  Then for all $t,s \in I$ we have
$$||(F\circ \gamma) '(t) -(F\circ \gamma) '(s)|| \leq  4\cdot \delta \cdot  \sqrt k  \,.$$
If $\gamma$ contains at least two points on one fiber of $F$ then for all $t\in I$
$$||(F\circ \gamma) '(t)||  \leq  6\cdot \delta  \cdot \sqrt k  \, .$$
\end{prop}

\begin{proof}
Let $p_i$ be the points such that $d_{p_i}$ is the $i$-th coordinate of $F$.
For  arbitrary $r>s \in I$, we see that  the $i$-th coordinate of
the differential $(F\circ \gamma) '(s)$,  is given  as minus the cosine of the angle $\angle p_i \gamma (s) \gamma (r)$.

For $r>s,t \in I$ we apply Lemma \ref{lem: almperp}, (1)  twice and deduce that the $i$-th coordinates of  $(F\circ \gamma )'(s)$
and of  $(F\circ \gamma) '(t)$ differ by at most $4\cdot \delta$. This implies the first statement.

From Lemma \ref{lem: almperp}, (2),  we see that $F(\gamma (s))=F(\gamma (r))$, for some $s<r \in I$,
 implies that
all coordinates of $(F\circ \gamma  )'(s)$ have absolute value at most $2\cdot \delta$.  Together with the first inequality, this implies the second one.
\end{proof}

\section{Fibers  of strainer maps}  \label{sec:fibers}
 \subsection{Local contractibility}


 \begin{thm}\label{prop: fibertop}
Let $F:\tilde  U \to \R^k$ be a distance map and $0<\delta \leq  \frac 1 {20 k}$. Assume that $F$ is a
$(k,\delta)$-strainer map at   $x\in U$  with straining radius  $\epsilon _x$.
  Set  $\Pi := F^{-1} (F(x))$.
Then,  for $W=B_{\epsilon _x} (x)$ there exists a homotopy  $\Phi :W\times [0,1] \to W$ retracting $W$ onto $W\cap \Pi$
with the following properties.
	\begin{enumerate}
 \item For all $y\in W$, the curve $\gamma  _y (t)  :=\Phi (y,t)$ starts in $y$ and  ends on
 $\Pi$.
\item The diameter of the curve $\gamma _y$ is bounded from above  by  $8\cdot k\cdot  d(F(y),F(x))$.
\item For all $t\in [0,1]$, we have  $d(\gamma _y(t) ,x) \leq d(x,y)$.
\end{enumerate}
\end{thm}

\begin{proof}
%
%
Let $f_i =d_{p_i}$ be the coordinates of $F$.
By definition of $\epsilon _x$ we find
 $q_i \in \tilde  U$ for $i=1,...,k$,  such that
$x$ lies on the geodesic $p_i q_i$ and
such that $(p_i), (q_i)$ are opposite $(k,2\delta)$-strainers
in $W$.  Set $a_i=f_i (x)$ and $\Pi_i := f_i^{-1} (a_i) \subset U$.


First,
we define  flows $\phi _i : W\times [0,1] \to U$.
For a point $y\in W$ with $f_i(y) \ge a_i$,
the flow  $\phi _i$ moves $y$ with velocity $1$
along the geodesic $yp_i$
until it reaches $\Pi _i$ and then the flow stops for all times.
For a point $y$ with $f_i(y) \le a_i$,
the flow moves $y$ along the
geodesic $yq_i$ until it reaches $\Pi _i$ and stops there.

Since $x$ is on the geodesic $p_iq_i$,
the $\CAT(\kappa)$ condition
 implies  that
the flow $\phi _i$ does not increase the distance to  $x$.

By the first  variation formula,
the value of $f_i$ changes along the flow lines of $\phi _i$
with velocity at least $1 - 2\delta$
until the point reaches the  set $\Pi _i$.
Moreover, for $j \neq i$, the value
of $f_j$ changes along the flow lines of $\phi _i$
with velocity at most $4\delta$.

Consider the function $M_i, M \colon U \to \R$
defined by
$$M_i (y) :=  |f_i (y)-a_i|   \; \; \text{and} \; \; M(y) := \max_{1 \le i \le k} M_i(y)\,.$$
Note that $M(y)=0$ if and only if $y\in  \Pi$  and $M(y) \leq d(F(y),F(x))$, for all $y\in U$.

The above observation shows that the flow line $\phi_i (y,t)$ reaches $\Pi_i$ at latest
at $t= (1-2\delta )^{-1} \cdot M_i (y)$.  Due to the first variation formula  and $\delta \leq \frac 1 {20 }$, we have
for all $j \neq i$ and all $1\geq t\geq  0$:
\begin{equation} \label{eq:mj}
M_j(\phi_i(y,t)) \le M_j(y) + \frac {4\delta}  { 1-2\delta } \cdot M_i(y) \leq M_j (y)+ 5 \delta\cdot M(y) \,.
\end{equation}

Consider the concatenation
$\Psi$  of the flows $\phi_1,...,\phi_k$.  Thus $\Psi :W\times [0,k] \to W$ is a homotopy which moves on the time interval
$[j-1,j]$ the point $\Psi (y,j-1)$ along the flow lines of $\phi _j$ to  $\Pi_j$.
 We apply $k$ times the inequality  \eqref{eq:mj}   and conclude
$$M (\Psi(y,t)) \leq (1+5\delta ) ^k\cdot M (y) \, ,$$
for all $(y,t) \in W\times[0,k]$.
By construction $M_j (\Psi (y,j))=0$. Applying \eqref{eq:mj} again, for all $j$, we improve the last inequality to
$$M(\Psi (y,k))\leq k\cdot 5\cdot \delta \cdot (1+5\cdot \delta )^k   \cdot M(y) \leq \frac 1 4 \cdot (1+\frac 1 {4k}) ^k \cdot M(y) \,.$$
Since $(1+ \frac 1 x) ^x$ is increasing and converges to the Euler number $e$, we   see
$$M(\Psi (y,k) ) \leq  \frac 1 4 \cdot e^{\frac 1 4}  \cdot M(y) < \frac 1 4 \cdot 2 \cdot M(y) = \frac 1 2 \cdot  M(y) \, .$$
Moreover, the flow line of the homotopy $\Psi$ of a point $y$ has length at most
$$ k\cdot \frac 1 {1-2\delta} \cdot (1+5\cdot \delta) ^k \cdot M(y) \leq 4\cdot k\cdot M(y) \,.$$

Putting the last two observations together, we  inductively arrive at the following conclusion about the
$m$-fold concatenation  $\Psi _m:W\times [0,k\cdot m] \to W$ of the homotopy $\Psi$.
For any $y\in W$, we have  $M(\Psi _m (y,k\cdot m)) \leq 2^{-m}  \cdot M(y)$. Moreover, the $\Psi_m$-flow line of $y$ has length at most $8k\cdot M(y)$.
Therefore, reparametrizing $\Psi_m$ we obtain a limit homotopy  $\Phi =\Psi ^{\infty}$ with the required properties.
\end{proof}
 As a  consequence we obtain:
\begin{proof}[Proof of Theorem \ref{thm:contractible}]
Due to Lemma \ref{lem:openstr}, the strainer map $F$ is open.

Under the assumptions of Theorem \ref{prop: fibertop}, let $V'\subset V$ be compact. Since the straining radius depends continuously on the point, we find some $\epsilon >0$ smaller than the straining radius at any $x\in V'$ and smaller than $d(V', \partial V)$. By Theorem \ref{prop: fibertop}, for   $x\in V'$ and  $r<\epsilon$
  the set $B_r(x)\cap F^{-1} (F(x))$  is a homotopy retract  of the contractible ball $B_r(x) \subset X$. Thus $B_r(x) \cap F^{-1} (F(x))$ is contractible.
\end{proof}


\subsection{Dichotomy}
The openness of strainer maps and local connectedness of their fibers implies a  dichotomy in the behavior of strainer maps. First a local result:

\begin{lem} \label{lem:dichot}
 Let $F$ be a $(k,\delta)$-strainer map at $x\in U$ with  $\delta \leq \frac 1 {20 \cdot k}$.  Let $3r$ be not larger than the
 the straining radius of $F$ at $x$.
Then either
\begin{itemize}
\item $F:B_{r} (x) \to \R^k$ is injective, or
\item  For all $y\in B_{r} (x)$ the fiber $\Pi:= F^{-1} (F(y) ) \cap B_{r} (y) $ is a connected set of   diameter
at least $r$.
\end{itemize}
\end{lem}

\begin{proof}
Fix $y\in B_{r} (x)$ and the fiber $\Pi:= F^{-1} (F(y) ) \cap B_{r} (y) $. Due to Theorem \ref{prop: fibertop}, $\Pi$ is connected. Assume that $\Pi $ is not a singleton.

If the diameter of $\Pi $ is smaller than $r$ we find a point $z\in \Pi$ which has in $\Pi$  maximal distance $s<r$ from $y$. Consider a point $z'$, such that
$z$ is on the geodesic   $\gamma=yz'$ with sufficiently small  $l:=d(z,z')$.

  Applying Proposition \ref{prop: differential}    we deduce that
  $||(F\circ \gamma )'(t)|| \leq 6 \cdot \delta \cdot \sqrt k $, for all $t$ in the interval of definition of
  $\gamma$.   Therefore, $||F(z') -F(z)|| \leq  6 \cdot \delta \cdot \sqrt k \cdot l$.

Since the map $F$ is  $2\cdot \sqrt k$-open, we find a point $z_0$ with
$$F(z_0)=F(z) \; \; \text{and} \; \;  d(z_0,z) \leq  2\cdot \sqrt k \cdot 6\cdot \delta \cdot    \sqrt  k\cdot l <l\, . $$
Therefore,
$d(y,z_0) >d(y,z)=s$.  If $l$ has been small enough, then $z_0$ is contained in $B_r(y)$ in contradiction to the choice of $z$.

Hence,  for any $y\in B_r(x)$, the fiber $\Pi _y =  F^{-1} (F(y) ) \cap B_{r} (y) $
is a connected set that is either a point or has diameter at least $r$.

Since the map $F$ is open, we deduce that the set of points $y$ at which  fiber $\Pi_y$  is a singleton is an open and closed subset of $B_r(x)$.  Therefore,  this set is  either empty or the whole ball $B_r(x)$. This finishes the proof.
\end{proof}

As a direct consequence of this local statement, the openness of strainer maps and a standard  connectedness argument  we get the following global statement:

\begin{prop} \label{prop:dichot}
For any
$(k,\delta)$-strainer map $F:V\to \R^k$ with $\delta \leq \frac 1 {20 \cdot k}$ and   connected, open  $V$ the following dichotomy holds true.
Either no fiber of $F$ in $V$ contains an isolated point, or all fibers of $F$  in $V$ are discrete.
\end{prop}

\subsection{Extendability of strainer maps}
 If a fiber of a strainer map is not discrete, the strainer map admits an extension to any small punctured ball in the fiber:
\begin{prop}
Let $F \colon V \to \R^k$
be a  $(k,\delta)$-strainer  map, let $p\in V$ be a point and let $\Pi$ be the fiber of $F$ through $p$.
Then there exists  some $r>0$ and a neighborhood
$W$ of  $(B_r(p) \setminus \{p \} \cap \Pi)$ in $V$ such that the map $\hat F=(F,d_p) :W\to \R^{k+1}$
is a $(k+1, 4\delta)$-strainer map.
\end{prop}

\begin{proof}
Let $r$ be smaller than the straining radius of $F$ at $p$ and such that the distance function $d_p$ is a
$(1,\delta)$-strainer in the punctured ball $B_r(p)\setminus \{p \}$.

For every $q\in (B_r(p)\setminus \{p \} )\cap \Pi$ and any of the $k$ points $p_i$ defining the map $F$,
we deduce from  Lemma \ref{lem: almperp}
$$ \frac \pi 2 - 2\delta<\angle pqp_i < \frac \pi 2 \, .$$
Thus, by Lemma \ref{lem:qi},   the tuple $(p_1,....,p_k,p)$ is a $(k+1,4\delta )$-strainer at the point $q$.
This proves the claim.
\end{proof}


\section{Finiteness results} \label{sec:strata}
\subsection{Notations}  As before, let $U\subset \tilde U \subset X$ be a tiny ball of radius $r_0\leq 1$
and capacity bounded by $N$. Let  $\delta >0 $ be arbitrary.

As in Subsection \ref{subsec:bilip}, we denote by $\mathcal A$ the distance sphere of radius $r_0$ around $U$
and by $\mathcal A _{ \delta  }$  a fixed maximal $ \delta   \cdot r_0$-separated subset of $\mathcal A$.
Let $m=m(N,\delta )$ be  an upper bound on the number of elements in $\mathcal A_{ \delta }$.

Let $k$ be a natural number. Denote by $\mathcal F _{ \delta }$ the set of distance maps $F :\tilde U \to \R^k$,  whose coordinates are
distance functions to  points $p_j\in \mathcal A_{ \delta }$. The number of elements in
$\mathcal F_{ \delta }$ is bounded from above by the constant $m^k$ depending  on $N$, $\delta$ and $k$.

\subsection{Bounding straining sequences}
For the investigations of $(k,\delta)$-strained points we may restrict the attention to
the finitely many maps from $\mathcal F_{\delta }$ :

\begin{lem}\label{lem: finite}
 Let $F:\tilde U \to \R^k$ be a distance map which is a $(k,\delta)$-strainer map at $x \in U$.
Then there exist maps $F_1,F_2 \in \mathcal F_{ \delta }$ such that the pairs $(F,F_2)$ and $(F_1,F_2)$ are opposite $(k,3\cdot \delta)$-strainer maps at $x$.
\end{lem}

\begin{proof}
Let $F$ be given by the $k$-tuple $(p_i)$.   Find an opposite $(k,\delta)$-strainer $(q_i)$ at the point $x$.
   By the  definition of $\mathcal A_{\delta}$, we find $k$-tuples  $(p_i')$ and $(q_i')$ in $\mathcal A_{ \delta }$
such that $$\angle p_i' x p_i  < \delta   \; \;  \text{and} \; \; \angle q_i' xq_i  <  \delta \,.  $$
Due to the triangle inequality and the definition of strainers,  the distance maps $F_1 ,F_2 \in \mathcal F_{ \delta}$ given by the
$k$-tuples $(p_i')$ and $(q_i')$
 have the required properties.
\end{proof}

\subsection{Bounding bad sequences}   First, a simple   lemma
(\cite[Lemma 10.3]{BGP}):

\begin{lem}  \label{lem:BGP}
For all  $N,L\geq 1$ and  natural number $M$ there exists  $K(N,M,L)>0$ with the following property.
 Let $Y$ be an $N$-doubling metric space. Then every subset $T$ of $Y$ with
 at least $K$ elements contains an $M$-tuple $(x_1,...,x_M)$ such that
$d(x_i,x_{i+1}) \geq L \cdot d(x_i,x_k)$, for all $1 \leq k \leq i \leq  M-1$.
\end{lem}

\begin{proof}
Fix $N,L \geq 1$. We find $C=C(N,L)$ such that any  set of   diameter  $D>0$  in any $N$-doubling  space is covered by at most $C$ subsets with diameter at most $\frac D {2\cdot L}$.

We are going to prove by induction on $M$ that $K(N, M, L)= C^{M-1}$ satisfies the claim of the lemma. The case $M=1$ is clear.

Assume the claim is true for $M-1$ and consider a subset $T$ of $Y$ with at
least $C^{M-1}$ elements.  Replacing $T$ by a finite  subset,  we can assume that the diameter $D>0$ of $T$ is finite.  Cover $T$ by at most $C$ subsets with diameter at most $\frac D {2\cdot L}$.  At least one of this subsets, say $T_1$, has  at least $ C^{M-2}$ elements. By the inductive assumption, we find a tuple
$(x_1,...,x_{M-1})$ in $T_1$ as in the statement of the lemma.

Take an arbitrary point $x_M$ in $T$ such that $d(x_M,x_{M-1}) \geq \frac D 2$.  By construction, the extended $M$-tuple $(x_1,...,x_M)$ satisfies the statement of the lemma.
\end{proof}

The following defines
 a counterpart of straining sequences:

\begin{defn}\label{defn: badseq}
A subset $T$  of   a tiny ball  $U$
is called \emph{$\delta$-bad}
if  no point $x\in T$ is a $(1,\delta)$-strainer of another point $y\in T$.
\end{defn}

We  derive the following uniform bound:

\begin{prop}\label{prop: bad}
There is a number $C_0=C_0(N,\delta)$ such that each $\delta$-bad subset of
 $U$ has at most $C_0$ elements.
\end{prop}

\begin{proof}
%

 The claim is scale invariant. Rescaling $U$, we may  assume that $r_0=1$. Hence  the curvature bound $\kappa$ is at most $\frac 1 {10}$.
 Moreover, we may assume $\delta \leq \pi$.

Using comparison of quadrangles, we find some $r_1>0$ depending only on  $\delta$ such that the following holds true for all
$y_1,x_1,x_2,y_2 \in  \tilde U$.   If $d(x_1,y_1)=d(x_2,y_2)=1$ and the angles satisfy $\angle y_1x_1x_2 \geq \delta /2$ and $\angle y_2x_2x_1  \geq \pi- \delta /4$ then the distance between $ y_1$ and $y_2$ is at least $r_1$.

We fix some number $L_0$ depending only on $\delta$ (and the curvature bound $\frac 1 {10}$),   such that for all
 triangles  $xyz$ in $U$, the inequality $d(x,z) \geq L_0\cdot d(x,y)$ implies
$\angle xzy \leq \frac \delta 4$.

Assume now that the Proposition does not hold. Then  there are arbitrary large   $\delta$-bad subsets, possibly in different tiny balls  $U$ (in different GCBA spaces),  but of the same bound on the  capacity $N$.

 By Lemma \ref{lem:BGP}, we then find
 $\delta$-bad sets $\{x_1,...,x_M\}$ with  arbitrary large   $M$, such that $d(x_i,x_{i+1}) \ge L_0 \cdot d(x_i,x_{k})$
for all $1 \leq k \leq i \leq  M-1$.

We fix this $M$-tuple $x_1,...,x_M$.
Denote by $v_{i,j} \in \Sigma _{x_i}$ the starting direction
of the geodesic $x_ix_j$.
For each $i \ge 2$, we use that $x_1$ is not a $(1,\delta)$-strainer at $x_i$ to find
 antipodes
$w_i^+, w_i^- \in \Sigma _{x_i}$ of $v_{i,1}$ such that
$d(w_i^+,w_i^-) \ge \delta$.

We proceed as follows.
For each $i \geq  3$ the distance in $\Sigma _{x_2}$
between $v _{2,i}$ and either $w_2^+$ or
$w_2^-$ is at least $\delta/2$.
Hence we can find a subsequence
$x_1, x_2, x_{l_3}, x_{l_4}, \dots, x_{l_k}$ of
the tuple $(x_i)$ with at least $M/2$ elements
such that
for one of the directions $w_2^{\pm}$, say $w_2^+$,
and for each $i \geq 3$ we have
$d(w_2^+ , v_{2,l_i}) \ge \delta/2$.
Denote this direction $w_2^+$ by $w_2$
and replace our original tuple $x_1,...,x_M$
by this subsequence.

We repeat  the procedure at $x_3$ and
continue inductively. In this way
we obtain a $\delta$-bad sequence  $x_1, \dots, x_s$
with $s \ge \log_2 M$ and, for each $i \geq 2$, a  direction $w_i\in \Sigma _{x_i}$,  such that
the following two conditions hold:

\begin{enumerate}
\item  $d(x_i,x_{i+1}) \ge L_0 \cdot  d(x_i,x_{k})$, for all $1 \leq k\leq  i < s$;
\item The  direction $w_i$  is antipodal to
$v_{i,1}$. For all $j > i$,
we have $d(v_{i,j},w_i) \ge \delta/2$.

\end{enumerate}

For $2 \le i \leq s$ choose a geodesic $\gamma _i$  in $\tilde U$
of length $1$ starting at $x_i$ in the direction $w_i$ and set $y_i=\gamma _i (1)$. Thus,
 $d(y_i,x_i)=1$.

Let $2\leq i<j \leq s$ be arbitrary.  By construction,
$\angle y_ix_ix_j \geq \delta/ 2$. On the other hand, by the choice of $L_0$,
we have $\angle x_1x_jx_i \le \delta/4$  and therefore, $\angle y_jx_jx_i \geq \pi -\delta /4$.
Due to the first statement in the proof, we have  $d(y_j,y_i) \geq r_1$.

Therefore, the doubling constant of $\bar B_1 (U)$ (and hence the capacity bound of $U$)  bounds the number $s$
in our sequence, providing a contradiction.
\end{proof}

\subsection{Extension of strainer maps}
We now prove the following central  result:

\begin{thm}\label{lem: +1strainer}
There exists $C_1=C_1(N,\delta) >0$  with the following properties.

Let $F \colon V \to \R^k$
be a $(k,\delta)$-strainer  map on an open subset $V$ of a tiny ball $U$ of capacity bounded by $N$.  Let $E$ denote the set of points in $V$ at which $F$ cannot be extended
to a  $(k+1, 12\cdot \delta)$-strainer map $\hat F=(F,f)$  using some distance function $f=d_{p_{k+1}}$ as  last coordinate.

Then $E$ intersects each fiber $\Pi$ of $F$ in $V$ in at most $C_1$ points.
$E$ is a countable union of compact subsets $E_j$, such that  the restriction $F:E_j \to F(E_j)$ is $C_1$-biLipschitz. Moreover,
\begin{equation} \label{eq:measurebound}
\mathcal H^k(E)\leq C_1 ^{k+1} \cdot \mathcal H^k (F(E)) \leq C_1^{2k+1} \cdot  10\cdot r_0 ^k \, .
\end{equation}
\end{thm}

\begin{proof}
   If $\delta >\frac \pi {12}$  then $E$ is empty, and  the statement is clear. Thus, we may assume $\delta \leq \frac \pi {12}$.
Due to Lemma \ref{lem:uniformbound}, there is a number $k_0=k_0 (N)$ such that $k\leq k_0$.

Let $F$ be defined by a $k$-tuple $(p_1,...,p_k)$.
By Lemma \ref{lem: finite}, there is a finite set $\mathcal F_{\delta}$ of distance maps $G:U\to \R^k$  with at most $C=C(N,\delta)$ elements  and the following property.
If $V_G$  denotes the  set of points in $V$ at which $F$ and $G$ are opposite $(k,3\cdot \delta)$-strainer maps, then  the open set
 $\bigcup \{ \, V_G \mid G \in \mathcal{F}_{\delta} \, \}$   contains $V$.
Since $\mathcal F_{\delta}$ has at most $C$ elements, we may replace $V$ by one of the sets $V_G$ and assume that
on the whole set  $V$ there exists   an opposite $(k,3\cdot \delta)$-strainer map $G$ to $F$.

Let $\Pi$  be a fiber of the map $F$ on $V$.
For any pair of points $x,y \in V\cap \Pi$ we deduce from Lemma \ref{lem: almperp}  that
 $|\angle p_i x y - \frac {\pi} 2|  < 6 \delta$.  Therefore, if $x$ were a $(1,6 \cdot \delta )$-strainer at $y$ then
 the $(k+1) $-tuple $(p_1,\dots,p_k,x)$ is a $(k+1,12 \cdot \delta )$-strainer at $x$, as follows from Corollary \ref{cor:complement}.

 Hence, the subset  $ E\cap \Pi$    must be $6\delta$-bad.
 Due to Proposition \ref{prop: bad},  $E\cap \Pi$ can have at most $C_0(N,6\cdot \delta)$ elements.
  This proves the first statement of the theorem.

 We claim,  for any sequence $x_l \in E$ converging to any  $x\in  E$, the inequality
\begin{equation} \label{eq:differ}
\liminf _{l\to \infty}  \frac {|| F(x)-F(x_l)||} {d(x,x_l)} \geq  \delta    \, .
\end{equation}
 Assume that \eqref{eq:differ} is violated.
Replacing $x_l$ by a subsequence and applying the first variation formula we deduce,  for any $i=1,...,k$ and all  large $l$, $|\angle p_ix x_l -\frac \pi 2| < 2\delta$.
   Fix an opposite $(k,\delta)$-strainer $(q_i)$ to $(p_i)$ at $x$. Then $(q_i)$ and $(p_i)$ are opposite $(k,\delta)$-strainers at $x_l$, for all
  $l$ large enough, Corollary \ref{cor: openopp}.  Applying     Lemma \ref{lem: almperp}, we deduce that  $|\angle p_i x_l x -\frac \pi 2| <4\delta$, for all sufficiently large $l$ and all $1\leq i\leq k$.
 But, due to Proposition  \ref{lem: first}, the point $x$ is a $(1,4\cdot \delta )$-strainer at $x_l$, for all $l$ large enough.
 Hence, $(p_1,\dots,p_k,x)$ is a $(k+1,8 \delta )$-strainer at $x_l$  (Corollary \ref{cor:complement}) in contradiction to the assumption $x_l \in E$.
 This finishes the proof of \eqref{eq:differ}.

The remaining claims are consequences of this infinitesimal property.
 We set $C_1:=\max \{\frac 4 {\delta}, 2\sqrt k_0, C_0 \}$, where  $k_0$ is a bound on $k$ and $C_0$ is a bound
on the number of elements of $E$ in fibers of $F$.
	The restriction of $F$ to $E$ is $2\sqrt k$-Lipschitz,  as any distance map.
	The set $E$ is closed in $V $, hence locally complete.
 The implication that $E$ is a union of compact subsets $E_j$ to which $F$ restricts as a $C_1$-biLipschitz map is shown in \cite[Lemma 3.1]{Ly-open},  as a consequence of \eqref{eq:differ}.

 The set $E$ is a union of a countable number of Lipschitz images  of compact subsets of $\R^k$, hence $E$ is countably
 $k$-rectifiable,  \cite{Kirchheim}.
 An application of  the co-area formula,  \cite{Kirchheim}, together with \eqref{eq:differ}
 proves the first inequality in \eqref{eq:measurebound}.  The second inequality in \eqref{eq:measurebound} follows from the fact that $F(E)$ is contained in  a Euclidean $k$-dimensional ball of radius
 $C_1\cdot r_0$, and the fact that the volumes of Euclidean unit balls in any dimension are smaller than $10$.

 This finishes the proof.
 \end{proof}

\subsection{Conclusions}
Note that Theorem \ref{lem: +1strainer} is a quantitative version of Theorem \ref{thm:technical}. Thus the proof
of  Theorem \ref{thm:technical} is finished as well.

In order to derive Theorem
\ref{thmfirst}, we prove the following localized more precise version of it.
Let again $U$ be a tiny ball of radius $r_0$ and capacity bounded by $N$ as above.
As  before, $U_{k,\delta}$  denotes the set of all $(k,\delta)$-strained points.

\begin{prop}\label{prop: rect}
 There exists a number $C_2=C_2(N,\delta ) >0$  with the following properties.  The set $U\setminus U_{k,\delta}$ is a union of countably many images of biLipschitz maps
 $G_j: A_j \to U$, with $A_j$  compact in $\R^{k-1}$. Moreover,
 $\mathcal H^{k-1}  (U\setminus U_{k,\delta} )< C_2 \cdot r_0 ^{k-1}$.
\end{prop}

\begin{proof}
If $\delta$ decreases,  the sets $U_{k,\delta}$ increase, thus in all subsequent considerations we may assume
that $\delta $ is sufficiently small.

We proceed by induction on $k$. The set   $U\setminus U_{1,\delta}$ has at most $C_0 (\delta, N)$ elements, due to Proposition \ref{prop: bad}. This  proves the statement for $k=1$.

Assuming the result is true for  $k$, we are going to prove it for $k+1$.
By the inductive assumption, the set $U\setminus U_{k, \delta /50 }$ is a  countable union of images of biLipschitz maps defined on compact subsets
of $\R^{k-1}$.

Thus it suffices to represent $K:=U_{k, \delta /50 } \setminus U_{k+1,\delta}$ as a union of biLipschitz images and to estimate its $k$-dimensional Hausdorff measure.

Any point $x\in K$ admits a $(k,\delta /12)$-strainer  map $F\in \mathcal F_{\delta /12}$,  due to Lemma \ref{lem: finite}.  Thus,
we have  a  finite number  of $(k,\delta /12)$-strainer  maps  $F_j:V_j\to \R^k$ defined on open subsets $V_j\subset U$
such that the  union of $V_j$ covers $K$ and such that the number of $V_j$ is bounded by some $C_3(N,\delta)$.

Applying now Theorem \ref{lem: +1strainer} to the maps $F_j:V_j \to \R^k$ and observing that $K_j:=K\cap V_j$
is contained in the set $E$ from the formulation of Theorem \ref{lem: +1strainer} we deduce the following.

Each $K_j$ is a countable union of biLipschitz images of compact subsets of $\R^{k}$ and $\mathcal H^{k} (K_j)$
is bounded by $C_4\cdot r_0 ^k$ for some $C_4=C_4(N,\delta)$.
 Summing up, we deduce the required bound on the volume $\mathcal H^k (K)$
and the fact that $K$ is a union of countably many images of biLipschitz maps defined on compact subsets of
$\R^k$.
\end{proof}

Now we obtain:
\begin{proof}[Proof of Theorem \ref{thmfirst}]
 We cover $X$ by a countable number of tiny balls $U$, using the separability of $X$.
  The set  $U\setminus X_{k,0} $ of not $(k,0)$-strained points  in  $U$ is the union of the complements $U\setminus   U_{k,\delta}$
 where $\delta$ runs over all sufficiently small  rational numbers.  Applying Proposition \ref{prop: rect}, we deduce that $X\setminus X_{k,0}$ is a countable union of compact subsets  biLipschitz  equivalent to subsets of $\R^{k-1}$.
\end{proof}

\begin{rem}\label{rem:remthmfirst}
Theorem \ref{thmfirst}  and Theorem \ref{thm:regular} strengthen  \cite[Main Theorem 1(2)]{Otsu},
stating that  $\mathcal{H}^n(X^n  \setminus X_{n,0}) = 0$ and that there exists a continuous Riemannian structure of
$X^n \cap  X_{n,0}$.
\end{rem}

\section{Dimension} \label{sec:dim}

\subsection{Topological and Hausdorff dimension}
We can now prove a quantitative version of the first part of Theorem \ref{thm0}.

\begin{prop} \label{lem:tophaus}
There is some $C (N)>0$  such that the following holds true for
any tiny ball $U$ of radius $r_0$ and capacity bounded by $N$.

If $n$ is the topological dimension of $U$ then $0<\mathcal H^n (U) < C\cdot r_0 ^n$.
In  particular, the Hausdorff dimension  of $U$ equals $n$.
Moreover, $n$ is the largest number such that some tangent cone $T_xU$ is isometric to $\R^n$.
Finally, $n$ is the largest number, such that there are $(n,\frac 1 {4\cdot n})$-strained points in $U$.
\end{prop}

\begin{proof}
We already know that the topological dimension $n$ of $U$ is finite. Then $\mathcal H^n(U) >0$, by general results in dimension theory, compare \cite{Edgar}.

By \cite{Kleiner}, the geometric dimension of $U$ is $\dim (U)=n$ as well.
Therefore, there are no points in $U$ at which the tangent space $T_xU$ contains an $(n+1)$-dimensional Euclidean space.
 Hence,
 $U$ is contained in $X\setminus X_{n+1,0}$.

Due to Theorem \ref{thmfirst}, $U$ is a countable union of biLipschitz images of subsets of $\R^n$. Therefore, the Hausdorff dimension
of $U$ is at most $n$.

Due to Lemma \ref{lem:openstr},   there are no $(n+1,\frac 1{4\cdot (n+1)})$-strainer maps defined on subsets of $U$. Thus, there are no $(n+1,\frac 1 {4\cdot (n+1)})$-strained points in $U$.
Due to Proposition \ref{prop: rect},  $ \mathcal H^n (U) < C\cdot r_0 ^n$,
for some $C$ depending only on $N$.

Applying  Theorem \ref{thmfirst} again, we find a point $x\in X$ such that the tangent space $T_x$ has $\R ^n$ as a direct factor.
If $T_x$ is not equal to $\R^n$ then it contains $\R^n\times [0,\infty )$. But this is impossible, since the  geometric dimension of $X$  is $n$. 
 Therefore, $T_x =\R^n$.

This finishes the proof.
\end{proof}

From now on, we fix some bound $n_0=n_0 (N)$ on the dimension on $U$ provided by Corollary \ref{lem:doubling} and set $\delta _0 = \delta _0 (N):= \frac 1 {50 \cdot n_0 ^2}$.
We can now relate the dichotomy observed in Proposition \ref{prop:dichot} to the dimension.
\begin{cor} \label{cor:dich}
Let $F:V\to \R^k$ be a $(k,\delta)$-strainer map  on a connected open subset $V$ of a tiny ball $U$.
If $\delta \leq \delta _0 (N)$  then one of the following possibilities occurs:
\begin{enumerate}
\item No fiber of $F$ in $V$ has isolated points. Then $\dim (W) >k$, for every open subset $W\subset  V$.
\item  $V$ is a $k$-dimensional topological manifold.  Then for every $x\in V$ and every $r$, such that $3r$ is smaller than
the straining radius of $F$ at $x$, the map $F:B_r(x)\to F(B_r(x))$ is $L$-biLipschitz, where $L$ goes to $1$ as $\delta$ goes to $0$.
\end{enumerate}
\end{cor}

\begin{proof}
By Proposition \ref{prop:dichot} either no fiber of $F$  has isolated points or the map $F$ is locally injective.

 In the second case, for any $x \in V$ and $r>0$  as in the statement above, we deduce from  Lemma \ref{lem:dichot} and  Lemma \ref{lem:openstr}  that $F:B_r(x) \to  F(B_r(x))$ is $L$-biLipschitz with $L=2\sqrt k$.
 Due to   Corollary \ref{cor:openstr}, we can choose $L$  close to $1$ if $\delta$ goes to $0$.
 Since $F(V)$ is open in $\R^k$, we see that $V$ is a $k$-dimensional manifold.

In the first case,   any $x\in V$   is a non-isolated point in the fiber $F^{-1} (F(x))$.
We apply Theorem  \ref{lem: +1strainer} and find  $(k+1,12\cdot \delta)$-strained  points arbitrary close to $x$.
 Then, by Proposition \ref{lem:openstr}, the dimension of any ball  around $x$  is at least $k+1$.
\end{proof}

Now we can finish

\begin{proof}[Proof of Theorem \ref{thm0}]
Given any GCBA space $X$, we cover $X$ by a countable number of tiny balls $U$ and reduce all statements to the case of tiny balls.
 For any  tiny ball $U$,
the topological dimension $n$ equals the Hausdorff dimension, by Proposition \ref{lem:tophaus}.
Moreover, by Proposition \ref{lem:tophaus}, there exists an $(n,\delta)$-strainer map $F:V\to \R^n$ for arbitrary small $\delta$
and some $V\subset U$.  Applying Corollary \ref{cor:dich}, we see that $V$ is a topological manifold. Hence, $n$ equals the maximal dimension of a
Euclidean ball which embeds into $U$ as an open set.
\end{proof}

\subsection{Lower bound on the measure}
 The Euclidean spheres are the smallest GCBA spaces with the same dimension  and curvature bound:
\begin{prop} \label{prop:mapsn}
Let $\Sigma $ be a compact GCBA space, which is $\CAT(1)$ and of dimension $n$.
  Then there exists a  $1$-Lipschitz surjection $P:\Sigma \to \mathbb S^n$.
\end{prop}

\begin{proof}
By Proposition \ref{lem:tophaus} we find a point $x\in \Sigma$ with $T_x$ isometric to $\R^n$.
Then one can define a surjective $1$-Lipschitz map
$P: \Sigma\to \mathbb  S^0 \ast \Sigma _x= \mathbb S^n$ as the "spherical logarithmic map",
i.e. the composition of the logarithmic map in $\Sigma $ and the exponential map in $\mathbb S^n$,
see \cite[Lemma 2.2]{Lbuild}.
\end{proof}

\begin{rem} \label{rem:volmin}
This observation is related to the \emph{volume minimality} of constant curvature spaces
proved in  \cite[Sections 6, 7]{N2} along with  rigidity statements.
\end{rem}

\subsection{Dimension and convergence} We are going to describe  the  behavior of dimension under convergence.
\begin{lem} \label{lem:dimconverge}
Let $\tilde U_l$ converge to $\tilde U$ as in  the standard setting for convergence. Let $x\in U$ be a limit point of $x_l\in U_l$.
  If $\dim (T_x)=n$ then there exists some $\epsilon >0$ and $l_0 \in \N$ such that for all
  $l\geq l_0$ the ball $B_{\epsilon} (x_l)$ has dimension $n$.

In particular, $\dim (T_{x_l})\leq n$, for all $l$ large enough.
\end{lem}

\begin{proof}
 First, assume   $\dim (B_{\epsilon } (x_l) ) <n$, for  some $\epsilon >0$ and infinitely many $l$.
Due to the semicontinuity of the geometric dimension under convergence (cf. \cite[Lemma 11.1]{Lbuild}), we conclude
$\dim (\bar B_r(x))<n$, for any $r<\epsilon$. But then
$\dim (T_x)<n$, by the definition of geometric dimension, in contradiction to our assumption.

Assuming that the statement of the lemma is wrong, we can therefore choose a subsequence and assume that  $\dim (B_{\frac 1 l} (x_l) ) =m+1>n$, for some fixed $m$ (since the dimensions in question are bounded,  Proposition \ref{prop:precomp}).  Since the dimension equals the geometric dimension, we find some $y_l\in B_{\frac 1 l} (x_l)$  with $\dim (\Sigma _{y_l} )=m$.

Due to Proposition \ref{prop:mapsn},  any $\Sigma _{y_l}$ and then also any limit space $\Sigma '$   of this sequence, admits a surjective $1$-Lipschitz map onto $\mathbb S^m$. Therefore, the Hausdorff dimension of $\Sigma'$ is at least $m$.  Due to  Lemma \ref{lem: semiproj}, $\Sigma_x$ admits a surjective $1$-Lipschitz map onto $\Sigma'$, since $y_l$ converge to $x$.
Hence the Hausdorff dimension of $\Sigma_x$ is at least $m$ as well. But this contradicts $\dim (T_x)=n \leq m$.

This contradiction finishes the proof.
\end{proof}




Let $X$ again be  GCBA. As in the introduction we consider the \emph{$k$-dimensional part $X^k$} of $X$  as the set of all points $x\in X$ with $\dim (T_x)=k$.
Applying Lemma \ref{lem:dimconverge} to the constant sequence $X_l=X$ we directly see:
\begin{cor} \label{cor:tangloc}
A point $x\in X$ is contained in $X^k$  if and only if there is some $\epsilon >0$, such that for all $r<\epsilon$ we have
$\dim (B_r (x))=k$. The closure of $X^k$ in $X$ does not contain points from $X^m$ with $m<k$.
\end{cor}

In the strained case we get more stability:

\begin{lem} \label{lem:stabled}
In the notations of Lemma \ref{lem:dimconverge}  above, assume  that
the point $x\in U$ is $(k,\delta)$-strained.
Then, for all sufficiently large $l$, we have $\dim (T_{x_l}) \geq k$.

 If,  in addition,  $\dim (T_{x_l} ) =k$ for all $l$ large enough,
then $n=k$. Hence $\dim (T_{x_l})= \dim (T_x)$, for all $l$ large enough.
\end{lem}

\begin{proof}
 We find some  $(k,\delta)$-strainer map $F$  in a neighborhood of $x$,  defined by a $k$-tuple $(p_i)$. We approximate this tuple by $k$-tuples in $\tilde U_l$ and obtain distance maps $F_l :\tilde U_l\to \R^k$ converging to $F$.  Moreover,  for all $l$ large enough, $F_l$ is a $(k,\delta)$-strainer map  at $x_l$ with a uniform lower bound
$3r$ on the straining radii of $F_l$ at $x_l$, Lemma \ref{lem:open} and Lemma  \ref{lem:strradius}.
Due to Lemma \ref{lem:openstr}, the dimension of any ball around $x_l$ must be at least $k$, hence $\dim (T_{x_l}) \geq k$.

Assume  $\dim (T_{x_l} )=k$, for all $l$ large enough. Due to  Corollary \ref{cor:dich}, the restriction of the strainer maps $F_l$ to the ball $B_r (x_l)$ is $L$-biLipschitz.   Therefore, so is the restriction of $F$ to $B_r(x)$. Applying   Corollary \ref{cor:dich} again, we see that
$B_r(x)$ is a $k$-dimensional manifold, hence  $n=k$.
\end{proof}

\subsection{Regular parts} \label{subsec:regpart}
   We fix now some $ \delta \leq \delta _0$.
   By the \emph{$k$-regular part} of $U$ we denote the set of $(k,\delta)$-strained  points $x\in U$ with $\dim (T_x)=k$.

\begin{cor} \label{cor:kreg}
Let $U$ be a tiny ball of radius $r_0$ and capacity bounded by $N$.  Let $k$ be a natural number.
 The set $Reg _k (U)$  of $k$-regular points is open in $U$, dense in $U^k$ and locally biLipschitz homeomorphic to
 $\R^k$.
The topological boundary $\partial Reg_k( U) :=U\cap  (\bar U^k \setminus Reg _k (U))$ of
$Reg _k (U)$  in   $U$ does not contain  $(k,\delta)$-strained points. Moreover,
$$\mathcal H^{k-1} (\bar U^k\setminus Reg _k (U)) <C\cdot r_0 ^{k-1}   \; \; \text{and} \;\; \mathcal H^k (U^k) < C\cdot r_0 ^k\, ,$$
for  some constant $C$ depending only on $N$ and the choice of $\delta $.
\end{cor}

\begin{proof}
Any point $x$ in $Reg_k (U)$ admits a $(k,\delta )$-strainer map $F$. Due to  Corollary  \ref{cor:dich}, the restriction of $F$ to a small ball
around $x$ is biLipschitz onto an open subset of $\R^k$.  Hence, this ball is contained in $U^k$ and consists of $(k,\delta )$-strained points.
Therefore, $Reg_k (U)$ is open in $X$ and locally biLipschitz to $\R^k$.

Let $x\in U^k$ be arbitrary.   By Corollary \ref{cor:tangloc},
any sufficiently  small ball $W$ around $x$ has dimension $k$. Hence, $W$  contains $(k,\delta )$-strained points, therefore points from $Reg_k (U)$. Thus, $Reg_k (U)$ is dense in $U^k$.

Assume that  $x\in \bar U^k$ is $(k,\delta )$-strained. Writing $x$ as a limit of points $x_l \in Reg_k (U)$ and applying Lemma \ref{lem:stabled},
we see $\dim (T_x) =k$. Hence $x\in Reg_k (U)$.


 No point  in $U^k$ is $(k+1,  12 \cdot \delta )$-strained, due to Lemma \ref{lem:openstr}. Thus the bounds on measures are
contained in Theorem \ref{lem: +1strainer}.
\end{proof}

\subsection{Conclusions}
We finish the proofs of  two  theorems from the introduction.

\begin{proof} [Proof of Theorem \ref{thm:locdim}]
Thus, let $X$ be GCBA  and $k$ a natural number.
 As we have seen in Corollary \ref{cor:tangloc},
a point $x\in X$  is in the $k$-dimensional part $X^k$ if and only if all sufficiently small balls around $x$ have dimension $k$.

Cover $X$ by a countable collection of tiny balls. For each of these  tiny balls $U$
consider  its $k$-regular part and let $M^k\subset X^k$ denote the union of these $k$-regular parts.  Due to Corollary \ref{cor:kreg}, this subset $M^k$ is open in $X$, dense in $X^k$  and
locally biLipschitz to $\R^k$. Moreover, $\bar X^k\setminus M^k$ is a countable union of  subsets of finite $(k-1)$-dimensional Hausdorff measure.

Every  nonempty $V\subset X^k$ which is open in $X^k$,  contains an open non-empty subset of $M^k$ hence $\mathcal H^k (V) >0$.  From Corollary \ref{cor:kreg}, we deduce that
the measure $\mathcal H^k (X^k\cap U)$ is finite for every tiny ball $U$.

 This finishes the proof.
\end{proof}

\begin{rem}\label{rem:remlocdim}
In  \cite[Main Theorem 1(1)]{Otsu} one finds the statement that $\mathcal H^k$ is locally positive and non-zero on $X^k$.
From \cite[Section 4]{Otsu} one can conclude  that  the set $X^k$ is a
Lipschitz manifold up to a subset of $\mathcal H^k$-measure $0$.
\end{rem}

Recall from the introduction that  the \emph{canonical measure} $\mu _X$ on $X$ is the sum over all $k=0,1,...$ of the restrictions of $\mathcal H^k$ to $X^k$.

\begin{proof} [Proof of Theorem \ref{thm:finitecan}]
Let $X$ again be GCBA. If $x\in X$ satisfies $\dim (T_x)=k$, thus $x\in X^k$,
then the measure  $\mathcal H^k \llcorner X^k$ is positive on any neighborhood $V$ of $x$, due to Theorem \ref{thm:locdim}. Hence $\mu_X(V)>0$.

On the other hand, the dimension of any tiny ball $U$ in $X$ is finite, hence only finitely many
of the measures $\mathcal H^k \llcorner X^k$ can be non-zero on $U$. Due to Corollary \ref{cor:kreg}, the measure  $\mathcal H^k (X^k\cap U)$ is finite, hence so is $\mu_X (U)$.

Therefore, the  measure $\mu _X$ is finite on any relatively compact subset of $X$.
\end{proof}

\section{Stability of the canonical measure} \label{sec:stablecan}
\subsection{Setting and preparations}
We are going to prove here Theorem \ref{thm:stablecan} and its local generalization.
First we recall the notion of measured Gromov--Hausdorff convergence, sufficient for our purposes,
compare \cite{HKST15} for details.

Let $Z_l$ be a sequence of compact spaces GH-converging to a compact set $Z$. Let $\mathcal M_l$ be a Radon measure on $Z_l$ and let $\mathcal M$ be a Radon measure on $Z$. \emph{The measures $\mathcal M_l$ converge to $\mathcal M$}  if for any compact sets $K_l\subset Z_l$ converging to $K\subset Z$
 the following holds true:
 \begin{equation} \label{eq:measure}
\lim _{\epsilon\to 0}  (\liminf _{l\to \infty}  \mathcal M_l (B_{\epsilon} (K_l)) =  \lim _{\epsilon\to 0}  (\limsup _{l\to \infty}  \mathcal M_l (B_{\epsilon} (K_l))   =  \mathcal M (K) \, .
\end{equation}

 By general results, any sequence of Radon measures $\mathcal M_l$ on $Z_l$ contains a converging subsequence if the total measures $\mathcal M_l (Z_l)$ are uniformly bounded.

We  continue working in the standard setting for convergence
as in Definition \ref{defn:standard}. We fix some $k=0,1....$ and restrict our attention to  the $k$-dimensional part $\mu ^k _U =\mathcal H^k \llcorner U^k $ of $\mu$.
The aim of this section is the following:
\begin{thm}  \label{thm:measconv}
Under the GH-convergence $\tilde U_l \to \tilde U$ the $k$-dimensional parts of the canonical measures $\mathcal M_l :=\mu ^k _{U_l}$ converge  to
$\mathcal M := \mu^k _U$ locally on $U$.  Thus, \eqref{eq:measure} holds for all compact $K\subset U$.
\end{thm}

We know that $\mathcal M_l (\bar U_l) $ is uniformly bounded  by a constant $C$, Corollary \ref{cor:kreg}. Thus, by general compactness of measures,
we may choose a subsequence and assume that the measures  $\mathcal M_l$ converge to a finite Radon measure $\mathcal N$ on $\bar U$. We need to verify that $\mathcal M=\mathcal N$ on $U$.

It suffices to prove that $\mathcal N$  coincides with $\mathcal H^k$ on the regular part  $Reg _k (U)$ and that $\mathcal N$ vanishes on the complement
$U\setminus Reg _k (U)$.  We fix $ \delta $ as in Subsection \ref{subsec:regpart}.

\subsection{Regular part} \label{subsec:convregular}
In order to prove that $\mathcal N$ and $\mathcal H^k$ coincide on the regular part  $Reg _k (U)$, we note that $\mathcal N$ satisfies
$\mathcal N (B_r (x)) \leq  C\cdot r^{k}$, whenever $\bar  B_r (x) \subset U$. Indeed,  this inequality is true for all the approximating measures, by Corollary \ref{cor:kreg}.  Thus,   $\mathcal N$ is  absolutely continuous  with respect to $\mathcal H^k$ on  the Lipschitz manifold  $Reg_k (U)$. By the Lebesgue-Radon-Nikodym differentiation theorem (see \cite{HKST15}),  it suffices to prove  that for $\mathcal H^k$-almost every point $x\in Reg_k (U)$ the  density
$$b(x):= \lim _{r\to 0} \frac {\mathcal N(B_r(x))} {\mathcal H^k (B_r (x))},$$
exists and is equal to $1$.

Due to Theorem \ref{thmfirst}, $\mathcal H^k$-almost every point in $Reg_k (U)$ has as tangent space $T_x =\R^k$.
Let $x$ be such a point and let $x_l$ be a sequence of points in $U_l$ converging to $x$.  We take points $p_1,...,p_k \in \tilde U$
such that the directions $(xp_i)'$ are pairwise orthogonal in  $T_x=\R^k$.  Then the distance map $F:U\to \R^k$ defined by the $k$-tuple $(p_i)$ is
a $(k, \rho  )$-strainer map  at $x$,  for any $\rho >0$.  Consider a sequence of  distance maps $F_l :U_l\to \R^k$ converging to $F$.

For any $\rho >0$, we find some $r>0$ and some $l_0>0$ such that $F$ and $F_l$, for $l\geq l_0$, are $(k,\rho )$-strainer maps with straining radius at least $3r$
at $x$ and $x_l$, respectively.
Then the maps $F :B_r(x)\to \R^k$ and $F_l:B_r(x_l)\to \R^k$ are $L$-biLipschitz onto their images and $L$ goes to $1$ as $\rho$ goes to $0$,  due to   Lemma \ref{lem:stabled} and
Corollary \ref{cor:dich}.
Moreover,  by Corollary \ref{cor:openstr}, the images contain
balls with radius $r/2$ around $F(x)$ and $F_l (x_l)$, respectively.

  Thus, for any $s<\frac r {10}$ and all sufficiently large  $l$,  the volumes
  $\mathcal H^k (B_s(x)), \mathcal H^k (B_s (x_l)) $ are bounded between $L^{-2k}  \cdot \omega _k \cdot s^k $ and
   $L^{2k}  \cdot \omega _k \cdot s^k $,
where $\omega _k$ denotes the volume of the $k$-dimensional Euclidean unit ball.

Since $L$ goes to $1$, as $\rho$ goes to $0$, we conclude $b(x)=1$.

\subsection{Singular part}
The support $S$ of $\mathcal N$ in $U$ is contained in the limit set of the supports of $\mu ^k _{U_l}$.   Thus,
$S$ is contained in the set of all
points $x\in U$, which are limits of a sequence of $k$-regular points $x_l \in  U_l$. Due to  Lemma \ref{lem:stabled},
any such point $x$ which is not in $Reg _k (U)$ cannot be $(k,\delta )$-strained.

Therefore, $T:=S\setminus Reg_k (U)$ is a closed subset of $U\setminus U_{k,\delta _1}$ of points which are not $(k,\delta )$-strained.
Note that $\mathcal H^k (T)=0$, by Theorem \ref{thmfirst}.  It is enough to prove that $\mathcal N(K)=0$, for any compact subset $K$ of $T$.

Fix a compact subset $K$ in $T$ and a sequence of  compact $K_l \subset U_l$ converging to $K$. Let finally $t>0$ be arbitrary.
It suffices to find some $s = s (t) >0$ such that
$$\mu ^k _{U_l} (B_{s } (K_l)) =\mathcal H^k (B_{s} (K_l) \cap Reg_k (U_l)) <t \, $$
  for all sufficiently large $l$.

%
As in Subsection \ref{subsec:bilip} denote by $\mathcal A_{\delta} \subset \tilde U$ some maximal $\delta  \cdot r_0$-separated subset in
the distance sphere  $\mathcal A$ of radius $r_0$ around $U$.  Numerate the elements of $\mathcal A_{\delta }$ as $\mathcal A_{\delta } =\{ p_1,...,p_m \}$
and  approximate any $p_i$ by points $p_i^l $ in the distance sphere of radius $r_0$ around $U_l$ in $\tilde U_l$.

For all $l$ large enough, the points $\{p_1^l,...,p_m^l\}$ are $\delta \cdot r_0$-dense in the  distance sphere of radius $r_0$ around $U_l$.
Denote by $\mathcal F_{\delta }$  the set of distance maps $F:\tilde U\to \R^k $ defined by  $k$-tuples in $\mathcal A_{\delta }$.
Denote by $\mathcal F_{\delta } ^l$ the corresponding lifts to distance maps $F_l:\tilde U_l \to \R^k$.
We numerate the elements of $F_{\delta}$ and $\mathcal F_{\delta} ^l$ as $G_1,....,G_j,...$ and $G_1^l,.....,G_j^l,....$, respectively.
These are finite sets (with $m^k$ elements). For any $j$, the distance maps $G_j^l$ converge to $G_j$.

The argument in Lemma \ref{lem: finite} shows, that for all $l$ large enough the following holds true:  If a point $x_l$ in $U_l$ is $(k,\delta _1)$-strained then there exists some  $G_j^l$ which is a $(k,3\cdot \delta)$-strainer map at $x_l$.

From Theorem \ref{lem: +1strainer}  and the finiteness of the elements in $\mathcal F _{\delta}$, we get  a number $C>0$  such that for any measurable subset $Y\subset Reg _k (U_l)$ we have  $$\mathcal H^k (Y) \leq   C\cdot  \max _j \mathcal H^k (G^l_j (Y)) \, .$$

Since all the maps $G_j^l$ are $2\sqrt k$-Lipschitz, the image $G_j^l  (B_s (K_l)  )$  is  contained in the $2\sqrt k\cdot s$-tubular neighborhood around $G^l_j(K_l)$.
Thus, for all $l$ large enough, $G^l_j  (B_s (K_l)  )$ is contained in the $3 \sqrt k\cdot s$-tubular neighborhood around $G_j (K)$. But $\mathcal H^k (K)=0$, hence
$\mathcal H^k (G_j (K))$  is $0$, for all $j$.  Thus, for all sufficiently small $s_0$, the $3\cdot \sqrt k\cdot s_0$-tubular neighborhood around
the compact set $G_j (K)$ has $\mathcal H^k$-measure less than $t$.

By the previous considerations, for such $s_0$ we have $\mu ^k  _{U_l} (B_{s_0}(K_l)) \leq C\cdot t$.  Since $t$ was arbitrary, this proves the claim.

\subsection{Conclusions}
We can now finish the
\begin{proof}[Proofs of Theorem \ref{thm:measconv} and Theorem \ref{thm:stablecan}]
The proof of Theorem \ref{thm:measconv} follows from the combination of the two Subsections above.

In order to prove Theorem \ref{thm:stablecan}, assume that $X_l$ are compact GCBA spaces with uniform bounds on dimension, curvature and injectivity radius. If $X_l$ converge in the GH-topology  to a space $X$ then
 the spaces $X_l$  are covered by a uniform number of uniformly bounded tiny balls, Proposition \ref{prop:precomp}. Thus the total measures $\mu _{X_l}(X_l)$ are uniformly bounded by Corollary \ref{cor:kreg}.
Hence, upon choosing a subsequence, we may assume that $\mu _{X_l}$ converges to a measure $\mathcal M$ on the limit space  $X$.
Applying the local statement of Theorem \ref{thm:measconv}, we see that $\mathcal M$ coincides  with  $\mu _X$.

Therefore, it remains to show that a uniform upper bound on the total measures $\mu _{X_l} (X_l)$ implies  precompactness of  the sequence $X_l$  in the GH-topology.

Assume the contrary. Then, applying Proposition \ref{prop:precomp}, we find tiny balls $U_l$ in $X_l$ of the same radius $r_0$ such that $\mu _{U_l} (U_l)$ converges to $0$.

Due to the uniform upper bound on the dimension, the 2-Lipschitz property of the logarithmic maps and Proposition \ref{prop:mapsn}, we find for any $s>0$ some $\epsilon >0$ such that the following holds true, for any $l$ and any $x_l\in U_l$. If  $T_{x_l} $ is $k$-dimensional and if the ball $B_{s}(x_l)$ is
contained in $U_l$ then $\mathcal H^k (B_s(x_l)) \geq \epsilon$.

 Set $s=\frac {r_0 } {2\cdot n }$, where $n$ is an   upper bound on the dimensions of $X_l$. With $\epsilon$ as above,   our assumption implies
 $\mu _{U_l} (U_l) <\epsilon$, for all $l$ large enough.
  By the above estimate,
   for any $k$ and any point $x \in U_l ^k$ in the $k$-dimensional
 part of $U_l$, the following holds true.  If $B_s(x)\subset U_l$ then $B_s(x)$ contains points from  $U_l^{k'}$
 with $k'>k$.

   Starting in the center of $U_l$, we  now construct by induction points $z_1,....,z_{n+1}$  in
   $U_l$, such that  $d(z_{i+1},z_i) <s$ and $\dim (T_{z_i})< \dim (T_{z_{i+1}})$, for all $1\leq i \leq n$.

   Thus, $\dim (T_{z_{n+1}}) >n$,  in contradiction to our assumption.
\end{proof}

\begin{rem}\label{rem:remstablecan}
Theorem \ref{thm:stablecan}
is a generalization of  \cite[Theorem 1.1]{N2}, where  Theorem \ref{thm:measconv}
is proved for the maximal   $k=\dim (X _l )=\dim (X)$.
\end{rem}

\subsection{Additional comments on the measure-theoretic structure of GCBA spaces}
Let $X$ be  a GCBA space and  $k$ a natural number.

 For any point $x\in X$ we consider a tiny ball $U$ around $x$ and apply
 Theorem  \ref{thm:measconv}   to  the convergence of rescaled spaces $(\frac 1 r \tilde U ,x) \to T_x$.
 We deduce that $\mu ^k$ has a well-defined \emph{$k$-dimensional density  at  the point $x$},  compare \cite[Theorem 1.4]{N2},
$$\lim  _{r\to 0} \frac {\mu ^k _U (B_r(x) ) } {r^k}  = \mu ^k_{T_x} (B_1(0))\, . $$

Let  now $x$ be  a point in  $X^k$  and let $\epsilon >0$ be as in Corollary \ref{cor:tangloc}.
 For every $z\in B_{\epsilon} (x) \cap X^k$
 and every $r<\epsilon  -d(x,z)$,
the measure $\mu ^k (B_r (z))$ is bounded from below by $ C(k) \cdot r^k $, due to the Lipschitz property of the
 logarithmic map and  Proposition \ref{prop:mapsn}. Here $C(k)$ is a positive constant depending only on $k$.
 Together with Corollary \ref{cor:kreg}, we see that  the restriction of $\mu ^k$
to $X^k$ is \emph{locally Ahlfors  $k$-regular}, see   \cite[Section 6]{N2} for related statements.

Finally,  for any open relatively compact subset  $U$ in any GCBA space $X$
 the Hausdorff dimension and the  \emph{rough dimension} coincide, see  \cite{Berest-path} for the definition of the rough dimension  (also known as
 the \emph{Assouad dimension}) and a discussion of this question.
Indeed, a thorough look  into  the proof of Proposition 10.4 and Theorem \ref{lem: +1strainer}  reveals the following claim,
for any tiny ball $U$:  For fixed $\delta >0$ and  for $\epsilon \to 0$,  any $\epsilon$-separated  subset
of $U\setminus U_{k,\delta}$ has  at most $O(\epsilon ^{1-k})$ elements.
 Hence,
  the \emph{rough dimension}
of  $U\setminus U_{k,\delta}$ is at most $k-1$.

\section{Homotopic stability} \label{sec:stable}



\subsection{Homotopic stability of fibers}
Let $U_l\subset \tilde U_l$ and $U\subset \tilde U$ be as in our standard setting for convergence, Definition \ref{defn:standard}.
We have the following general stability result:

\begin{thm} \label{thm:fiberstab}
Under the standard setting for convergence, let the distance maps $F_l:\bar U_l \to \R^k$ converge to the distance map $F:\bar U\to \R^k$.
Assume that the restriction of $F$ to an open set $V\subset U$ is  a $(k,\delta)$-strainer map,  with $\delta \leq \frac 1 {20 \cdot k}$.  Let $\mathfrak t_l \to \mathfrak t$ be a converging sequence in $\R^k$ and assume that the fiber
 $\Pi:=F^{-1} (\mathfrak t) \subset V$ is compact.  Let, finally, $K_l\subset U_l$ be compact sets converging to $\Pi$.

Then there exists $r>0$ such that the following holds true, for all $l$ large enough.  The restriction of $F_l$ to $V_l =B_r (K_l)$ is a $(k,\delta)$-strainer map, the fibers $\Pi_l := F^{-1} _l (\mathfrak t_l)\subset V_l$ are compact and  converge to $\Pi$.

Finally, $\Pi _l$ is homotopy equivalent to $\Pi$, for all $l$ large enough.
\end{thm}

\begin{proof}
By the compactness of $\Pi$, we find some $r>0$, such that for any $x\in \Pi$ the straining radius of $x$ with respect to $F$ is larger than $2r$.
Due to  Lemma \ref{lem:open},   $F_l$ is a $(k,\delta)$-strainer in $V_l$, for all $l$ large enough.
 Moreover, for all $l$ large enough
and any $x_l\in V_l$,  the straining radius of $F_l$ at $x_l$ is at least $r$, Lemma  \ref{lem:strradius}.

The maps  $F_l$ are $2\sqrt k$-open on $V_l$. This   implies that $\Pi_l$  converges to $\Pi$.

   For all $l$ large enough,  all balls in $\Pi_l$  of all radii $s<r$
are contractible, due to  Theorem \ref{prop: fibertop}.  The homotopy equivalence of $\Pi_l$ and $\Pi$  is now a direct consequence of the general homotopy stability theorem  \cite[Theorem A]{Petersen}.
\end{proof}

 We discuss two special cases. The first one  is an immediate application of the theorem in the case of a constant sequence $X^l=X$.

 \begin{cor}
Let $F:V\to \R^k$ be a $(k,\delta)$-strainer map defined on an open subset $V$ of the tiny ball $U$. Assume that a fiber $\Pi$ of $F$ is compact and  $\delta \leq  \frac 1 {20 \cdot k}$. Then all fibers of $F$, sufficiently close to $\Pi$, are homotopy equivalent to $\Pi$.
 \end{cor}

As a second application we obtain:
\begin{proof}[Proof of Theorem \ref{thmsph}]
Indeed, for any point $x$ in a GCBA space $X$, we find a tiny ball  $U\subset \tilde U$  containing $x$. Consider an arbitrary sequence of positive  numbers $t_l$ converging $0$ and   corresponding metric  spheres
$\partial B _{t_l} (x)$.

 Consider  the convergence $( \frac 1 {t_l} \tilde U ,x) \to (T_x,0)$, provided by  Corollary \ref{cor:converge}.  In the tangent cone $T_x$ the origin is a $(1,\delta)$-strainer
at any point of $T_x \setminus \{ 0 \}$, for any $\delta >0$.
 Moreover, the fiber $F^{-1} (1)$  of the distance function $F$ to the vertex of the cone is exactly $\Sigma _x \subset T_x$. The sphere $\partial B_{t_l} (x) $ is
 the fiber $d_x ^{-1} (1)$ of
 the distance functions $d_x$ on
$\frac 1 {t_l} \cdot \tilde U$. Thus, by Theorem \ref{thm:fiberstab}, $\partial B_{t_l} (x) $ and $\Sigma _x$ are homotopy equivalent, for all $l$ large enough.
Since the sequence $(t_l)$ was arbitrary, this finishes the proof.
\end{proof}

Using in addition Lemma \ref{lem:tangent},  one could   observe, that a homotopy equivalence in Theorem \ref{thmsph} is provided by the logarithmic map.

\subsection{Homotopy types of spaces of directions}
We are going to discuss a homotopy   property of the spaces of directions. Due to Theorem \ref{thmsph} this also determines the homotopy types of small distance spheres.

In  $\CAT(1)$ spaces all balls of radius less than $\pi$ are contractible. Applying the stability theorem  \cite[Theorem A]{Petersen} and Proposition \ref{prop:precomp}, we see
that on any compact set $\mathcal C$ of isometry classes
 of compact, geodesically complete $\CAT(1)$ spaces the following holds true.  For any $\Sigma \in \mathcal C$ the set of spaces in $\mathcal C$ with the homotopy type of $\Sigma $ is open and closed in $\mathcal C$.
From Lemma \ref{lem:equiv} we immediately derive:
\begin{cor}
For every $N>0$ there exists some $\epsilon (N)>0$ such that for any
tiny ball $U$ of capacity bounded by $N$, for any natural $k$,
  any point  $x\in U_{k,\epsilon }$,  the space of directions $\Sigma _x$
 is homotopy equivalent to a $k$-fold suspension.

 In particular, only for finitely many points $x\in U$ the space of directions $\Sigma _x$   is not
  homotopy equivalent to a suspension of some other space.
\end{cor}

\section{The differentiable structure} \label{sec:DC}
This section is essentially a rewording of \cite{P3}. The only result without a direct analogue in \cite{P3} is Proposition \ref{prop:contin}.

\subsection{Setting}  We are going to prove Theorem \ref{thm:regular}. The statement is local, so we may restrict ourselves to a tiny ball $U$  and assume that $U$ coincides with its set of $k$-regular points.  Hence we may assume that $U$ is a $k$-dimensional manifold and
 that every point $x\in U$ is $(k,\delta)$-strained, for a sufficiently  small
 $\delta =\delta (k)  \leq \frac 1 {50 \cdot k^2} $.

 \subsection{Euclidean points} \label{subsec: euclid}
 For any point $x\in U$, a neighborhood of $x$ in $U$ is biLipschitz to a $k$-dimensional Euclidean ball. Therefore, $T_xU$ is biLipschitz to $\R^k$.
 If $T_xU$ is a direct product $T_xU =\R^{k-1} \times Y$ for some space $Y$  then $Y$ must be $\R$.  (Indeed, $Y$ must be a $1$-dimensional cone over a finite set. By homological considerations,    it must be the cone over a two-point space).
 Therefore, a point $x\in U$  whose tangent cone $T_x$ has $\mathbb R^{k-1}$ as a direct factor satisfies $T_x=\R^k$.

 We call $x\in U$ with $T_x=\R^k$  a \emph{Euclidean point}. Let $\mathcal R =\mathcal R_k$   denote the set of all Euclidean points in $U$.
  From Theorem \ref{thmfirst} and the previous conclusion,  we deduce that
 $U\setminus \mathcal R$ has Hausdorff dimension at most $k-2$.

\subsection{Charts, differentials, Riemannian metric} \label{subsec:chart}
Let now $V\subset U$ be an open, convex subset and assume that $F$ and $G$ are opposite $(k,\delta )$-strainer maps in $V$.
 Due to Corollary \ref{cor:dich}, the map $F:V\to \R^k$ is locally $L$-biLipschitz with $L\leq 2\sqrt k$. Moreover, $L$ goes to $1$ as $\delta$ goes to $0$. Thus, for any $x\in V$, the differential $D_xF:T_x\to \R^k$ is an $L$-biLipschitz map.

 If $F(x)=F(y)$ for some $x,y \in V$  then, by Proposition \ref{prop: differential},
 $$D_xF(v) \leq  6\cdot \delta \cdot \sqrt k < \frac 1 {3\cdot \sqrt k} \, ,$$
where $v \in \Sigma _x\subset T_x$ is the starting direction of the geodesic $xy$.    This contradicts the fact that $D_xF$ is $L$-biLipschitz.
Therefore, $F:V\to F(V)$ is injective.

The preimage $F^{-1}$ is (directionally) differentiable at
all points of $\hat V:=F(V)$ with differentials being the inverse maps of the differentials of $F$. This differentiability just means, that compositions of any differentiable curve with $F^{-1}$ has well-defined directions in all points,
 compare \cite{Diff}.

For any   Euclidean point $x\in \mathcal R$, the differential $D_xF :T_x \to \R^k$ is a linear map,  again
by  the first variation formula. We have the following continuity property in $\mathcal R$.  Let $v_l \in \Sigma _{x_l}$  be the starting direction of a geodesic $\gamma _l$.
Let $x_l \in \mathcal R$ converge to a point $x\in \mathcal R$ and the geodesics $\gamma _l$  converge to a geodesic $\gamma$ with  starting direction
$v\in T_x$.  Then $D_{x_l} F(v_l)$ converge in $\R^k$ to $D_xF (v)$.  Indeed, this is just the reformulation of the statement that
angles between $x_l p_i$, for $i=1,...,k$, and $\gamma _l$ converge to the angle between $p_i x$ and $\gamma$. The last statement
can be easiest seen as a consequence of Lemma \ref{lem: semiproj} and the fact that the concerned spaces of directions are all unit spheres.

We  call the image $\hat V=F(V)$ together with the map $F : V\to  \hat V$   a \emph{metric chart}. On this metric chart, we have the subset $\hat {\mathcal R} := F(\mathcal R )$ whose complement in $\hat V$ has Hausdorff dimension at most $k-2$.  For any point $y\in \hat {\mathcal R}$ we get  a scalar product on its tangent space $T_y \R^k$, given by the pullback (via the linear map $D_yF^{-1}$)  of the scalar product on $T_{F^{-1} (y)} U$.  Due to the previous considerations, this Riemannian metric  $g_F$ is  continuous on $\hat {\mathcal R}$.

Expressing the length of a Lipschitz curve as an integral of its pointwise velocities, we see that for any Lipschitz curve
$\gamma$ in $\mathcal R$ the length of $\gamma$ coincides with the length of $\bar \gamma := F\circ \gamma$ with respect to the Riemannian metric $g_F$, hence
$$\ell (\gamma )= \int |\bar \gamma ' (t)| _{g_F} \, dt \, .$$

\subsection{$\DC$-maps in Euclidean spaces}
We refer the reader to \cite{P3} and \cite{Ambr} for more details.

A function $f:V\to \R$ on an open subset $V$ of $\R^m$ is called a \emph{$\DC$-function} if in a neighborhood of each point $x\in V$
one can write $f$ as a difference of two convex functions.  The set of  DC-functions contains all functions of class $C^{1,1}$ and
it is closed under addition and  multiplication.

A map $F:V\to \R^l$ is called a DC-map if its coordinates are DC.  The composition of DC-maps is again a DC-map.
In other words, a map $F:V\to \R^l$ is DC if and only if for every  DC-function $g:W\subset \R^l \to \R$, the composition
$g\circ F $ is a DC-function on $F^{-1} (W)$.

\subsection{$\DC$-maps on metric spaces}
The following definition is meaningful only if the metric spaces in question are (locally) geodesic.

\begin{defn}\label{defn: dc}
Let $Y$ be  a metric space.
A function $f:Y\to \R$ is called
a \emph{$\DC$-function}
if it can be locally represented as difference of two Lipschitz continuous  convex
functions.
\end{defn}

Due to the corresponding statements about DC-functions on intervals,  the set
 of $\DC$ functions on $Y$ is closed under addition and
multiplication.

\begin{rem}
We refer to \cite{Petruninsemi}
for the definition  and properties  of semi-convexity.
Assume that in  $Y$ each point $x$ admits  a Lipschitz $1$-convex function in a small neighborhood $V$ of $x$.
Then each semi-convex function on $Y$ is $\DC$.
Such strongly convex functions exist on Alexandrov spaces with lower curvature bound, \cite{Petruninsemi}.
On any  $\CAT(\kappa)$ space $X$ we get such a function as a scalar multiple of  $d_x^2$.
\end{rem}

We use compositions to define DC-maps between metric spaces.

\begin{defn} \label{defn: dcmap}
A locally Lipschitz map $F \colon Z \to Y$ between metric spaces $Z$ and $Y$
is called a \emph{$\DC$-map}
if for each $\DC$-function $f \colon U \to \R$
defined on an open subset $U$ of $Y$
the composition
$f \circ F$ is $\DC$ on $F^{-1}(U)$.
If $F$ is  a biLipschitz  homeomorphism  and its inverse is
$\DC$,
then we call  $F$  a \emph{$\DC$-isomorphism}.
\end{defn}

 A composition of DC-maps is a DC-map.  For a map $F:Z\to \R^l$ we recover the
old definition:  $F$ is DC if and only if the coordinates of $F$ are DC.


\subsection{Crucial observation}
Let now $U\subset \tilde U\subset X$ be again a tiny ball consisting of $k$-regular points as above.
  Since all  distance functions to points in $\tilde U$ are convex,
each $(k,\delta)$-strainer map is a DC-map by definition.  These strainer maps turn out to be DC-isomorphisms, in direct analogy with \cite{P2}, see also \cite{Ambr}. The proof of the following observation is  taken from \cite{P2}.

\begin{prop}\label{prop: dcisom}
Let $F$ and $G$ be opposite $(k,\delta)$-strainer maps in an open subset $V$ of a tiny ball $U$. Then $F:V\to F(V)\subset \R^k$
is  a $\DC$-isomorphism, if $\delta \leq    \frac 1 {50 \cdot k^2}$.
\end{prop}

\begin{proof}
Denote by $f_i=d_{p_i}$ the coordinates of $F$.
We already know that the map $F$ is a locally biLipschitz DC-map.
It remains to prove that the
inverse map $F^{-1} \colon F(V) \to V$
is  $\DC$ too. Thus, given an open subset $O \subset V$
 and a  convex function $g \colon O \to \R$,
we have to show that the function
$\overline{g} = g \circ F^{-1}$ is $\DC$ on $F(O)$.



We introduce the following auxiliary notion.
 We say that a  convex, Lipschitz continuous function $g:O\to \R$  on an open subset $O\subset V$ is \emph{$\alpha$-special} for some $\alpha \geq 0$ if the following holds true.
 For any $x\in O$ and any unit vector $v\in T_x$ such that $D_x f_i (v) \geq 0$ for all $i=1,...,k$  we have $D_xg (v) \leq -\alpha$.

If $g$ is $\alpha$-special then, for any Lipschitz curve $\eta :[a,b] \to O$ parametrized by arclength and such that all $f_i$ are non-decreasing on $\eta$, the composition $g\circ \eta :[a,b] \to \R$ decreases at least with velocity $\alpha$.

The proof of the Proposition will  follow from two auxiliary statements:

\begin{lem} \label{lem:aux1}
There is a $1$-Lipschitz $\alpha$-special function $g$ on $V$ with $\alpha =\frac 1 {4\cdot k^2}$.
\end{lem}

\begin{lem} \label{lem:aux2}
If $g$ is a $0$-special function in $O$ then the composition $\bar  g = g\circ F^{-1}$ is a convex function on $F(O)$.
\end{lem}

Indeed, assuming Lemma \ref{lem:aux1} and Lemma  \ref{lem:aux2} to be true,  we  derive:
\begin{cor} \label{cor:aux}
In the notations above let $h:O\to \R$ be an $L_0$-Lipschitz convex function.
Then $h$ can be represented as $h=h_1 -h_2$ with $h_1$ and $h_2$ being $0$-special
$L_0 \cdot (1+\frac 1 {\alpha} )$-Lipschitz functions.

Moreover, $\bar h := h\circ F^{-1}$ is the difference of two $C \cdot L_0$-Lipschitz convex functions, with some $C$ depending only on $k$.
\end{cor}

\begin{proof}
 Indeed, choosing $g$ as in Lemma \ref{lem:aux1}, we set $h_2(x):= \frac  {L_0} {\alpha} \cdot g(x)$.
 Then $h_2$ is $L_0$-special. Since $h$  is convex and  $L_0$-Lipschitz, we deduce that the function  $h_1 =g +h_2$ is convex and  $0$-special.
 The statement about the Lipschitz constants of $h_1$ and $h_2$ is clear.

  Due to Lemma \ref{lem:aux2}, the compositions $\bar h_{i}:= h_{i} \circ F^{-1}$ are convex on $F(O)$.  The Lipschitz constants of $\bar h_{i}$ are bounded from above by the product of the Lipschitz constants of $h_{i}$ and $F^{-1}$.
\end{proof}

Thus assuming  Lemma \ref{lem:aux1} and Lemma \ref{lem:aux2}  to be true,  we finish the proof of the proposition.
\end{proof}

We turn to  the auxiliary lemmas used in Proposition \ref{prop: dcisom}.

\begin{proof} [Proof of Lemma \ref{lem:aux1}]
Let $f_i$ be the coordinates of $F$ and let $g_i$ be the coordiates of $G$. The functions  $g_i$ are convex  and $1$-Lipschitz, for $i=1,...,k$,
hence so is $g (x) =\frac 1 k \sum _{i=1} ^k g_i (x)$.
We claim that  $g$ is $\frac 1 {4\cdot k^2} $-special.

Indeed, let $x\in V$ be arbitrary and let $v\in \Sigma _x$ be such that $D_xf_i (v) \geq 0$, for all $i=1,...,k$.
Then $D_x g_i (v) < \delta$ for all $i=1,...,k$, as follows directly from the first variation formula and the definition
of opposite strainer maps.

The map $D_xG :T_x\to \R^k$ is $2\sqrt k$-biLipschitz, thus $D_xG (v)$ has norm at least $\frac 1 {2\sqrt k}$.
 Therefore, for at least one $1\leq j \leq k$, we must have
 $|D_xg_j (v) |\geq \frac 1 {2  k}$.

For this $j$ we get, $D_xg_j (v) \leq -\frac 1 {2 k} $.  Summing up, we  obtain
$$D_xg (v) \leq  \frac 1  k \cdot ( -\frac 1 {2 k}  +  (k-1)\cdot \delta)  \leq   \frac 1 k \cdot (- \frac  1 {2 k} + \frac 1 {4\cdot k} )  \leq   - \frac 1 {4\cdot k^2}  \,.$$
This finishes the proof.
\end{proof}

\begin{proof} [Proof of Lemma \ref{lem:aux2}]
 It follows word by word as in \cite{P3}, since the proof in \cite{P3} only uses convexity and differentiability and
 no   curvature bounds.
\end{proof}

\subsection{The Riemannian metric revisited}
As in \cite{P3}, we have:

\begin{lem} \label{lem: bv}
For any metric chart $F:V\to \R^k$ as above, the Riemannian metric $g_F$ defined and continuous on the subset
$\hat {\mathcal R} =F(\mathcal R)$ is locally  of bounded variation.
Moreover, $g_F$ is differentiable almost everywhere in $F(\mathcal R)$.
\end{lem}

The proof  literally follows from \cite{P3} (see also \cite{Ambr}). The  idea is to take a
sufficiently large  set of  generic points $q_j$ in $\bar U$. The distance functions $h_j$ to  these points
have the following property. The compositions $\bar h_j :=h_j\circ F^{-1}$ are DC-functions by Proposition \ref{prop: dcisom}.
On the other hand, since $h_j$ are distance functions, the gradients of $h_j$ at all points  of $\hat {\mathcal  R}$ have  norm $1$ with respect to the Riemannian metric $g_F$.  One obtains an equation for the coordinates of $g_F$ and shows that they can be expressed  through the first derivatives of the DC-functions $\bar h_j$.

\subsection{$\DC$-curves in GCBA spaces}
In order to prove that  the Riemannian structure on the set $\mathcal R$ determines the metric,  we will need a stability statement about variations of DC-curves,
which might be of independent interest.   In the following definition and Proposition \ref{prop:contin}
we work in general GCBA spaces, and  not only in their regular parts as in the rest  of this section.

Let  $U \subset \tilde  U$ be a tiny ball. We say that a curve $\gamma :I\to  U$ on a compact interval $I$  is
 a \emph{$\DC$-curve of norm bounded by $A$} if  $\gamma$ is $A$-Lipschitz and for any $1$-Lipschitz convex function $f: U \to \R$ the
 restriction $f\circ \gamma$ can be (globally) written as a difference of two $A$-Lipschitz convex   functions on $I$.


The following statement  is closely related to the well-known fact \cite{AR}, that the length is continuous under convergence of curves of uniformly bounded turn in the Euclidean space.
\begin{prop} \label{prop:contin}
Let $\gamma_l:I\to  U$ be $\DC$-curves with a uniform bound on the norms. If $\gamma _l$ converges to $\gamma$ pointwise then
 $\lim_{l\to \infty} \ell (\gamma _l) =\ell (\gamma )$.
\end{prop}

\begin{proof}
Assuming the contrary and choosing a subsequence, we   find $\epsilon >0$ with
$$(1+2\cdot \epsilon) ^2\cdot \ell (\gamma ) < \lim _{l\to \infty} \ell (\gamma _l) \, .$$
   Due to Proposition \ref{lem:embed}, we find a  distance map
     $F_{\epsilon}: U\to \R^m _{\infty}$, which is
 a $(1+\epsilon)$-biLipschitz embedding, if $\R^m$ is equipped with the sup-norm $ | \cdot  |_{\infty} $.    Set $\eta _l=F_{\epsilon} \circ \gamma _l$ and $\eta =F _{\epsilon} \circ \gamma$.  From the biLipschitz property
 we obtain a contradiction, once we show that the lengths of $\eta _l$
converge to the length of $\eta$ in $\R^{m} _{\infty}$.

 The $i$-th coordinate of  $\eta _l $ is the composition of $\gamma_l$ and a convex distance function $d_{p_i}$. Thus,
 this  $i$-th coordinate  is a
 difference of two convex $A$-Lipschitz functions $h_l^+$ and $h_l^-$ on $I$. Adding a constant we may assume that $h_l^+ $ equals $0$ at some fixed point on $I$.

  Going to subsequences, we may assume that $h_l^+$ and $h_l^-$ converge  to $h^+$ and $h^-$ such that $h^+-h^-$ is the corresponding coordinate of $\eta $.
 Due to the standard results about convergence of convex functions, we see that at almost every $t \in I$, the differentials of $h_l^+, h_l^-$ exists at $t$ and converge to the differentials of $h^+,h^-$ at $t$.  Taking again all coordinates together, we see that for almost every $t\in I$, the differentials  $\eta _l' (t) \in \R^m$ exist and  converge to $\eta'(t)$.

 Expressing the length of $\eta$ and $\eta_l$ as integrals of $|\cdot |_{\infty}$-norms of $\eta'$ and $\eta_l'$ over $I$ we finish the proof of the convergence.  This finishes the proof of the Proposition.
\end{proof}

 Coming back to the regular part,  we can use this result to prove:
\begin{cor} \label{cor: compl}
Let $F:V\to \R^k$ be a metric chart as in  Subsection \ref{subsec:chart}, with convex $V\subset U$.
Let $S$ be a   subset of $V$ with $\mathcal H^{k-1} (S)=0$.
Then every pair of points $x,y\in V\setminus S$ is connected in $V\setminus S$ by
curves of lengths arbitrary close to $d(x,y)$.
\end{cor}

\begin{proof}
The statement is well-known and easy to prove for open convex subsets $\hat V$  in $\R^k$,
connecting $x$ and $y$ by concatenations of two segments.

 Since the  statement is true in $F(V)$ and  the map $F :V\to \hat V=F(V)$ is biLipschitz, it suffices to prove the following claim.
 Let $\gamma:I\to V$ be a geodesic. Then there exist curves $\gamma _l :I\to V\setminus S$ converging to $\gamma$ and such that
 $\ell (\gamma _l)$ converges to $\ell (\gamma)$.
 (Once such $\gamma _l$ are constructed we obtain, the desired curves by connecting the endpoints of $\gamma _l$ with $x$ and $y$ within $V\setminus S$,  using that $F$ is biLipschitz).

 In order to find such $\gamma _l$ we consider the curve $\eta:= F\circ \gamma $ in $\hat V$. Note that the differentials of $\eta$ at different
 points have distance at most  $2\cdot k\cdot \delta$ from each other, as follows from
 Lemma \ref{lem: almperp}.  Take  a small ball $B$ around the origin in the hyperplane  of $\R^k$ orthogonal to the starting direction of $\eta$. Then we observe that the map $Q:B\times  I\to \R^k$ given by $Q(x,t)= x+ \eta (t)$ is a biLipschitz embedding.

 This implies that for almost every $x_0\in B$ the curve $t\to \eta (t) +x_0$ does not meet  the set $F(S)$ with vanishing $\mathcal H^{k-1}$-measure.  Letting $x_0$ going to $0$, we find a sequence of translates  $\eta _l (t) =\eta (t) +x_l$ converging to $\eta$ and disjoint from $F(S)$.

 We set $\gamma _l =F^{-1} \circ \eta _l$. It suffices to prove that  $\ell (\gamma _l)$ converge to $\ell (\gamma)$.

Clearly, the curves $\gamma _l$ are uniformly Lipschitz.  Let  $f$ be a convex $1$-Lipschitz function on $V$.
We have $f\circ \gamma  _l =  f \circ F^{-1} \circ \eta _l$.

Due to Corollary  \ref{cor:aux}, $f\circ F^{-1}$ is the difference of two convex $A$-Lipschitz functions $h_1$ and $h_2$ on $F(V)$, where
$A$ is independent of $f$.  On the other hand, the curve $\eta$ is a DC-curve of bounded norm, since its coordinates are convex $1$-Lipschitz functions.  The curves $\eta_l$ are then also DC-curves with the same bound on the norm.  Together, this implies that $f\circ \gamma _l$ can be written as a difference of two convex $B$-Lipschitz functions, with some $B$ independent of $l$.

Hence $\gamma _l$ are DC-curves of uniformly bounded norm and the claim follows from Proposition \ref{prop:contin}.
\end{proof}


\subsection{Conclusions}
Now we can summarize the results to the
\begin{proof} [Proof of Theorem \ref{thm:regular}]
Define as above $M^k$ to be the set of all $(k,\delta)$-strained points in the $k$-dimensional part $X^k$,
with   $\delta \leq \frac 1 {50\cdot k^2}$. We have seen in Theorem \ref{thm:locdim}, that $M^k$ is a Lipschitz manifold.  By construction, every point in $\mathcal R=\mathcal R_k$ with tangent space isometric to $\R^k$ is contained in $M^k$.

For any open convex set $V$ with opposite $(k,\delta)$-strainer maps $F,G:V\to \R^k$, the map $F:V\to F(V)$
  is a DC-isomorphism onto an open subset of $\R^k$, Proposition \ref{prop: dcisom}. Thus, the set of all such charts provides $M^k$ with a DC-atlas.

On the set of Euclidean points $\mathcal R$ in  $M^k$ we get a Riemannian metric $g_F$ in any chart.
   Moreover, $M^k\setminus \mathcal R$ has Hausdorf dimension at most $k-2$ as shown in Subsection \ref{subsec: euclid}.
 Due to the intrinsic definition, this metric is globally well defined on $\mathcal R$. As shown in Subsection \ref{subsec:chart}, the Riemannian tensor is continuous on $\mathcal R$ and due to Lemma \ref{lem: bv}, it is locally of bounded variation.

The length of all curves contained in $\mathcal R$ is computed via the Riemannian metric. Finally, the length of all curves in $\mathcal R$ locally determines the metric in $M^k$, due to Corollary \ref{cor: compl}.
\end{proof}

\subsection{Second order  differentiability of $\DC$-functions}
  Let $X$ be an arbitrary GCBA space. By Theorem \ref{thm:locdim},  $\mu _X$-almost
all of $X$ is the union of the (different dimensional) regular parts  $M^k$ of $X$.  Due to Theorem \ref{thm:regular},
$\mu_X$-almost every point of $X$ is a point with a Euclidean tangent space.

 Applying the classical theorem of Rademacher in the metric charts of the regular part, we deduce that every  Lipschitz function $f:X\to \R$ has a linear differential $\mu_X$-almost everywhere, (in the sense of Stolz, as
 in Proposition \ref{prop:hessian} below.)

As in the case of Alexandrov spaces described in  \cite{P3}, all DC-functions   are almost everywhere twice differentiable, as stated in the following Proposition.

\begin{prop}  \label{prop:hessian}
 Let $X$ be a GCBA space and
 let $f:X\to \R$ be a $\DC$-function.
 Then, for $\mu _X$-almost all $x$, there exists a  bilinear form $H_x =H_x(f) :T_x\times T_x \to \R$, called the \emph{Hessian} of $f$ at $x$, such that the following holds true for any tiny ball $U$ around $x$.
   The remainder
  $R_x :U\to \R$  in the \emph{Taylor formula}
 \begin{equation} \label{eq:hess}
R_x(y) := f(y)- (f(x)+ D_xf (v) + H_xf (v,v) )\,  ,
\end{equation}
 where  $v:=\log _x (y)$, satisfies
 \begin{equation} \label{eq:hess1}
\lim _{y\to x} \frac {R_x(y)} {d(x,y)^2} =0 \, .
\end{equation}
\end{prop}

We only sketch the  proof and  refer for  details to  \cite{P3} and \cite[Section 7.2]{Ambr}.

 The claim is local and $\mu_X$-almost all of $X$ consists of regular points. Thus, we may replace $X$
by a tiny ball $U$ which coincides with its set of regular points $U=M^k$. Now we can use the DC-structure
provided by Theorem \ref{thm:regular}.

Using a coordinate change to "normal coordinates"  as in \cite{P3} and  \cite{Ambr},  Proposition \ref{prop:hessian}  follows  directly from the corresponding theorem of Alexandrov in $\R^n$,
\cite[Theorem 6.9]{Evans}, once the following lemma is verified. In the formulation of the lemma  and later on, we denote by $o$ as usual the \emph{Landau symbol}.

\begin{lem}
Let $G:V\to \R^k$ be a DC-isomorphism on an open subset $V\subset U$, given by a composition of a metric chart $F$ and a diffeomorphism of $\R^k$.  Let $x\in \mathcal R$ be a Euclidean point with $G(x)=0$.
  Assume that  the metric tensor $g$  of $V$ expressed on $W=G(V)$ via $G$ satisfies,
for all $y\in V\cap \mathcal R$,
\begin{equation} \label{eq:almeuc}
||g (G(y)) -g (G(x))|| =o(d(x,y))\, .
\end{equation}
   Then,    for  all $y\in V$
and the corresponding direction
$v=\log_x (y) \in T_x$, we have
$$||G(y) -D_xG (v)||   =o(d(x,y) ^2)\, .$$
\end{lem}

\begin{proof}
 We sketch the  proof, referring for details  to  \cite{P3} and \cite[Section 7.2]{Ambr}.

  From  \eqref{eq:almeuc}, and the fact  that the Riemannian tensor on $\mathcal R$ determines the metric in $V$,
 Corollary \ref{cor: compl},
 we obtain,  for all small $r$, and all $y,z\in \bar B_r(x)$,  the estimate
 \begin{equation} \label{eq:r2}
 | \, d(y,z) -||G(y)-G(z)|| \,| =o(r^2) \, .
 \end{equation}
 Hence, it suffices to prove, that for all $y\in \bar B_r (x)$ the angle $\beta (y)$  between $G(y)$ and $D_x G (v)$,
 (with $v=\log_x (y)$ as in the formulation) satisfies
 the estimate $\beta (y) = o(r)$.

  In order to prove this estimate, it is sufficient to show that for the midpoint $m$ of the geodesic $xy$ the angle
  $\beta _1 (y)$ between $G(y)$ and $G(m)$ satisfies $\beta _1 (y) =o(r)$.
  (Relying only on   \eqref{eq:r2}, one can show $\beta_1 (y)=o(\sqrt r)$ and as a consequence that
$\beta (y)= o(\sqrt r)$, as done in the course of the proof  of \cite[Proposition 7.8 (d)]{Ambr}.)
 In order to prove the required stronger estimate $\beta _1 (y)= o(r)$, we will  rely on the curvature
 bound, similarly to \cite{P3}.

 We say  that the triangle $xyz$ in $\bar B_r (x)$  is \emph{sufficiently  non-degenerated}, respectively \emph{very non-degenerated},  if all of its comparison angles are  at least $\frac \pi {100}$, respectively at least $\frac \pi {10}$.   For any  sufficiently non-degenerated    triangle $xyz$  in $\bar B_r (x)$, we deduce
 from  \eqref{eq:r2},  that  the comparison angle $\tilde \angle yxz$ differs from the angle between
 $G(y)$ and $G(z)$  in $\R^k$  by at most $o(r)$.

 Given a very non-degenerated triangle $xyz$ in $\bar B_r (x)$, we  find a point $w \in \bar  B_r (x)$ such that the triangles  $xyw$ and $xzw$ are
 sufficiently non-degenerated and such that
 $$\angle yxz +\angle yxw +\angle wxz = 2\pi  \, .$$
 Since the corresponding comparison angles  are not smaller and since the three angles between pairs of different vectors in $\{G(y),G(z), G(w) \}$ sum up to at most
 $2\pi$, we arrive at the following conclusion:

 For any very non-degenerated triangle $xyz$ in $\bar B_r (x)$ the
 angle $\angle yxz$ differs from  the angle  in $\R^k$ between $G(y)$ and $G(z)$
 by at most $o(r)$.

 Let now $y\in \bar B_r (x)$ be arbitrary and let $m$ be the midpoint of $xy$. We find a point $z$ with $d(x,y)=d(x,z)$,
  such that  $G(z)$ lies in the same plane as $G(y)$ and $G(m)$ and such that $G(z)$ is orthogonal to $G(y)$.
  Then the difference of the angle between $G(z)$ and $G(y)$ and the angle between  $G(z)$ and $G(m)$ is exactly the angle between
  $G(y)$ and $G(m)$.  On the other hand, due to the previous considerations, the angle between $G(z)$ and $G(y)$ (respectively,
  between $G(z)$ and $G(m)$)
  coincides with $\angle zxy$  (respectively, with $\angle zxm$)  up to $o(r)$. But, by construction, $\angle zxy=\angle zxm$.

  Therefore, we have verified the estimate $\beta _1 (y)=o(r)$, thus
  finishing the proof of  the Lemma and of  Proposition \ref{prop:hessian}.
\end{proof}

\begin{rem} \label{rem:hessian}
The second order differentiability of distance functions, a special case  of Proposition \ref{prop:hessian},
 appears in  \cite[Main Theorem 1(3)]{Otsu} .
\end{rem}


\section{Topological counterexamples} \label{sec:misc}

\begin{exmp} \label{ex:graph}
Let $X_n$ denote a unit circle $S$ with two other unit circles $S^{\pm}_n$
attached to $S$ at points  $p_n^{\pm}$
  at distance $ 1/n$ from each
other.  The sequence $X_n$ converges to the wedge of three unit circles $X$. Thus,  $X_n$ is not homeomorphic to $X$
for no $n$. This shows that there is no topological stability even
in dimension $1$.
\end{exmp}

\begin{exmp} \label{ex:surface}  Proposition \ref{lem: first} implies that $1$-dimensional
GCBA spaces are locally isometric to finite graphs. On the other hand,
Kleiner constructs in  \cite{Kleiner}  a
$2$-dimensional GCBA space $X$
that  does not admit a  triangulation, see also
 \cite[Example 2.7]{N1}.
 This space $X$ contains  a point $x$, such that no neighborhood of it is homeomorphic to a cone.
 Moreover, there are arbitrary  small
$r_1,r_2 >0$ such that the distance spheres $\partial B_{r_1} (x)$ and $\partial B_{r_2} (x)$ are not homeomorphic.
\end{exmp}


\bibliographystyle{alpha}
\bibliography{LGC}

\end{document}